\documentclass[11pt]{amsart}

\usepackage[a4paper]{geometry}
\usepackage{latexsym, amsthm, amsfonts, amsmath,amssymb,graphicx,lmodern,mathrsfs}

\usepackage[latin1]{inputenc} 
\usepackage[T1]{fontenc} 
\usepackage[english]{babel} 

\newtheorem{theo}{Theorem}[section]
\newtheorem{lem}[theo]{Lemma}
\newtheorem{propo}[theo]{Proposition}
\newtheorem{coro}[theo]{Corollary}
\newtheorem*{theointro}{Theorem}
\theoremstyle{definition}
\newtheorem{defi}[theo]{Definition}

\theoremstyle{remark}
\newtheorem{rem}[theo]{Remark}
\newtheorem{ex}[theo]{Example}

\makeatletter

\@addtoreset{equation}{section}
\makeatother
\def\R{\mathbb{R}}
\def\Z{\mathbb{Z}}
\def\C{\mathbb{C}}
\def\N{\mathbb{N}}

\def\z{\zeta}
\def\n{\eta}
\def\r{\rho}

\def\a{\alpha}
\def\e{\varepsilon}
\def\d{\delta}

\def\t{\theta}
\def\b{\beta}

\def\n'{\nu}
\def\d{\delta}
\def\l{\lambda}

\def\k{\kappa}

\def\D{\Delta}
\def\G{\Gamma}
\def\L{\Lambda}
\def\T{\Theta}
\def\P{\Phi}

\def\dq {\delta_{q}}
\def\sq {\sigma_{q}}

\begin{document}
\sloppy
\title{$q$-deformation of meromorphic solutions of linear differential equations.}
\author{Thomas Dreyfus}
\address{Université Paul Sabatier - Institut de Mathématiques de Toulouse,}
\curraddr{118 route de Narbonne, 31062 Toulouse}
\email{tdreyfus@math.univ-toulouse.fr}
\thanks{Work supported by the labex CIMI}



\date{\today}

\begin{abstract}
In this paper, we consider the behaviour, when $q$ goes to $1$, of the set of a convenient basis of meromorphic solutions of a family of linear $q$-difference equations. In particular, we show that, under convenient assumptions, such basis of meromorphic solutions converges, when $q$ goes to~$1$, to a basis of meromorphic solutions of a linear differential equation. We also explain that given a linear differential equation of order at least two, which has a Newton polygon that has only slopes of multiplicities one, and a basis of meromorphic solutions, we may build a family of linear $q$-difference equations that discretizes the linear differential equation, such that a convenient family of basis of meromorphic solutions is a $q$-deformation of the given basis of meromorphic solutions of the linear differential equation.
\end{abstract} 

\maketitle
\tableofcontents

\pagebreak[3]
\section*{Introduction}
Let~$q>1$ be a real parameter, and let us define the dilatation operator~$\sq$ 
$$\sq \big(f(z)\big):=f(qz).$$ 

When~$q$ tends to~$1$, the~$q$-difference operator~$\dq:=\frac{\sq-\mathrm{Id}}{(q-1)}$  ``tends'' to the derivation $\d:=z\frac{d}{dz}$. Hence every differential equation may be discretized by a~$q$-difference equation. Given a linear differential equation~$\widetilde{\D}$, and a basis of meromorphic solutions of $\widetilde{\D}$, we wonder if we can build $\D_{q}$, family of linear $q$-difference equations that is a $q$-deformation\footnote{Throughout the paper, we will say that the family of objects $\left(X_{q}\right)_{q>1}$ is a $q$-deformation of the object $\widetilde{X}$, if $X_{q}$ converges, in a certain sense, to $\widetilde{X}$, when $q\rightarrow 1$.} of $\widetilde{\D}$, and a convenient family of basis of meromorphic solutions of $\D_{q}$, that is a $q$-deformation of the basis of meromorphic solutions of $\widetilde{\D}$. The goal of this paper is to give an answer to this problem. 
\begin{center}
$\ast\ast\ast$
\end{center}
Let us consider 
$$\left\{\begin{array}{lll}
\dq Y(z,q)&=&B(z)Y(z,q)\\\\
\d \widetilde{Y}(z)&=&B(z)\widetilde{Y}(z),
\end{array}\right.
$$
where~$B(z)$, is a $m$ by $m$ square matrix with coefficients that are germs of meromorphic functions at $z=0$. We are going to recall the main result of \cite{S00} in the particular case where the above matrix $B(z)$ does not depend upon $q$ and $q>1$ is real. In \cite{S00}, Sauloy assumes that the systems are Fuchsian at~$0$ and the linear differential system has exponents at~$0$ which are non resonant (see \cite{S00},~$\S 1$, for a precise definition). 
The Frobenius algorithm provides a local fundamental solution, i.e, an invertible solution matrix, at~$z=0$,~$\widetilde{\P}_{0}(z)$, of the linear differential system~${\d \widetilde{Y}(z)=B(z)\widetilde{Y}(z)}$. This solution can be analytically continued into an analytic solution on~$\C^{*}$, minus a finite number of lines and half lines of the form~${\R_{>0}\a :=\Big\{ x\a \Big| x\in ]0,\infty[ \Big\}}$ and~${\R_{\geq 1}\b :=\Big\{ x\b \Big| x\in [1,\infty[ \Big\}}$, with~$\a,\b\in \C^{*}$. Note that in Sauloy's paper, the lines and half lines are in fact respectively $q$-spirals and $q$-half-spirals since the author considers the case where $q$ is a complex number such that $|q|>1$.\par
 In \cite{S00},~$\S 1$, the author uses a~$q$-analogue of the 
Frobenius algorithm to construct a local fundamental matrix solution at~$z=0$,~$\P_{0}(z,q)$, of the family of linear~$q$-difference systems ${\dq Y(z,q)=B(z)Y(z,q)}$, which is for a fixed~$q$, meromorphic on~$\C^{*}$
and has its poles contained in a finite number of~$q$-spirals of the form~${ q^{\Z}\a:=\left\{q^{n}\a , n\in \Z\right\}}$ and~${q^{\N^{*}}\b :=\left\{  q^{n}\b, n\in \N^{*} \right\}}$, with~$\a,\b\in \C^{*}$.  Sauloy proves that~$\P_{0}(z,q)$ converges uniformly to~$\widetilde{\P}_{0}(z)$ when~$q\to 1$, in every compact subset of its domain of definition. \\ \par 
The problem in the non Fuchsian case is more difficult. Divergent formal power series may appear as solutions of the linear differential systems, but we may apply to them a Borel-Laplace summation process in order to obtain the existence of germs of meromorphic solutions on sectors of the Riemann surface of the logarithm. See  \cite{B,Ber,LR90,LR95,M95,MR,R93,RM1,S09,VdPS}.  The same situation occurs in the $q$-difference case. See  \cite{Be,Bu,DSK,DVRSZ,DVZ,D4,D3,MZ,R92,RS07,RS09,RSZ,RZ,S04,S04b,Trj,vdPR,Z99,Z00,Z01,Z02,Z03}. Let us give more precisions on \cite{RSZ}. We refer to $\S \ref{sec2}$ for more details. The authors of \cite{RSZ} consider a linear $q$-difference system having coefficients that are germs of meromorphic functions, and having integral slopes. See \cite{RSZ}, $\S 2.2$, for a precise definition. In this case, the work of Birkhoff and Guenther implies that after an analytic gauge transformation, such system may be put into a very simple form, that is in the Birkhoff-Guenther normal form. Moreover, after a formal gauge transformation, a system in the Birkhoff-Guenther normal form may be put into a diagonal bloc system. Then, the authors of \cite{RSZ} build a set of meromorphic gauge transformations, that make the same transformations as the formal gauge transformation, with germs of entries, having poles contained in a finite number of~$q$-spirals of the form $q^{\Z}\a$, with $\a\in \C^{*}$. Moreover, the meromorphic gauge transformations they build are uniquely determined by the set of poles and their multiplicities.
\begin{center}
$\ast\ast\ast$
\end{center}
The paper is organized as follows. In $\S \ref{sec1}$, we make an overview of the local study of the linear differential equations. In particular, we remind how a meromorphic linear differential operator may be factorized, with formal power series, or with germs of  meromorphic functions on some sectors of the Riemann surface of the logarithm. Given a linear differential equation, with coefficients that are germs of meromorphic functions, we also remind the existence of basis of solutions that are germs of meromorphic functions on some sectors of the Riemann surface of the logarithm. In $\S \ref{sec2}$, we summarize the work of \cite{RSZ}. We recall how they attach to a convenient linear $q$-difference system, a set of meromorphic fundamental solution. In $\S \ref{sec3}$, we explain how, under convenient assumptions, we may express the meromorphic solutions of \cite{RSZ}, using Jackson integral, that is a $q$-discretization of the classical integral. In $\S \ref{sec4}$ we consider $q$ as a real parameter we make converges to $1$. 
In $\S \ref{sec41}$ we prove a preliminary result of confluence\footnote{Throughout the paper, we will use the word ``confluence'' to describe the $q$-degeneracy when $q\rightarrow 1$.}. See Theorem \ref{theo2} for a more precise statement.

\begin{theointro}
Let us consider a linear differential equation~$\widetilde{\D}$ and a family of linear~$q$-difference equations~$\D_{q}$ that discretizes~$\widetilde{\D}$. Then, under convenient assumptions, a family of basis of meromorphic solutions of $\D_{q}$ given by \cite{RSZ}, converges, when $q$ goes to $1$, to a basis of meromorphic solutions of $\widetilde{\D}$.
\end{theointro}

We apply the above result in $\S \ref{sec43}$,  in a particular example, and in $\S \ref{sec42}$, where we prove our main result. See Theorem \ref{theo3} for a more precise statement and \cite{VdPS}, $\S 3.3$, for the  definition of the Newton polygon.

\begin{theointro}
Let us consider a linear differential equation~$\widetilde{\D}$ of order at least two, assume that its Newton polygon has only slopes of multiplicities one, and let us fix a certain basis of meromorphic solutions of $\widetilde{\D}$. Then, there exists $\D_{q}$, family of linear~$q$-difference equations, that is a $q$-deformation of $\widetilde{\D}$, and there exists a family of basis of meromorphic solutions of $\D_{q}$ given by \cite{RSZ}, that is a $q$-deformation of the given basis of meromorphic solutions of $\widetilde{\D}$.
\end{theointro}

Note that we construct explicitly the family $\D_{q}$, and the family of basis of meromorphic solutions of $\D_{q}$.
Remark also that confluence problems in the non Fuchsian case were considered in the papers \cite{DVZ,D3,Z02}, but this article is, to the best of our knowledge, the first to give a confluence result of the solutions built in \cite{RSZ} in the case where the linear $q$-difference equations have several slopes. We refer to Remarks~\ref{rem1} and~\ref{rem2} for a rough statement of the results in \cite{DVZ,D3}. Note also that the results of \cite{DVZ} are deeply used in $\S \ref{sec42}$, in order to construct the family $\D_{q}$. \\ \par 

\textbf{Acknowledgments.}
The author would like to thank Jean-Pierre Ramis and the anonymous referees, for their suggestions to improve the quality of the paper. 

\pagebreak[3]
\section{Local study of linear differential equations}\label{sec1}
In this section, we make a short overview of the local formal and analytic study of linear differential equations. See \cite{VdPS} for more details.

\subsection{Local formal study of linear differential equations}\label{sec11}

Let~$\C[[z]]$ be the ring of formal power series,~$\C((z)):=\C[[z]][z^{-1}]$ be its fraction field,~$\d:=z\frac{d}{dz}$, and consider a monic linear differential equation in coefficients in $\C((z))$:
\begin{equation}\label{eq10}
\widetilde{P}(\widetilde{y})=0.
\end{equation}
 Once for all, we fix a determination of the complex logarithm over $\widetilde{\C}$ we call~$\log$. For $a\in \C$, we write $z^{a}:=e^{a\log(z)}$. 
As we can see in \cite{VdPS}, Theorem 3.1, there exist $\widetilde{f}_{1},\dots,\widetilde{f}_{m}\in \displaystyle\bigcup_{\nu\in \N^{*}}\C\left(\left(z^{1/\nu}\right)\right)$, such that we have the factorization
$$\widetilde{P}=\left(\d -\widetilde{f}_{m}\right)\dots\left(\d -\widetilde{f}_{1}\right).$$
Moreover, for all $1\leq j<k\leq m$, $v_{0}\left(\widetilde{f}_{j}\right)\leq v_{0}\left(\widetilde{f}_{k}\right)$, where $v_{0}$ denotes the $z$-valuation. 

It follows that (\ref{eq10}) is equivalent to
\begin{equation}\label{eq13}
\d \widetilde{Y}=\widetilde{C}\widetilde{Y}, \hbox{ where }
 \widetilde{C}:=\begin{pmatrix}
\widetilde{f}_{1} &1&&0\\
&\ddots&\ddots&\\
&&\ddots&1\\
0&&&\widetilde{f}_{m}
\end{pmatrix}.
\end{equation}

\begin{center}\textbf{Until the end of the section, we are going to assume that $\widetilde{f}_{1},\dots,\widetilde{f}_{m}\in \C\left(\left(z\right)\right)$.}\end{center}

The goal of this subsection is to give an explicit form of a fundamental solution for~(\ref{eq13}) in coefficients in a field we are going to introduce now.\\ \par 

Let 
$$\textbf{E}:= z^{-1} \C \left[z^{-1}\right].~$$
We define formally the~differential ring,~$$\textbf{R}:=\C((z))\left[\log, \left(z^{\widetilde{a}}\right)_{\widetilde{a} \in \C},\Big(e\left(\widetilde{\l}\right)\Big)_{\widetilde{\l} \in \textbf{E}} \right],$$
 with the following rules:
\begin{enumerate}
\item The symbols~$\log$,~$\Big(z^{\widetilde{a}}\Big)_{\widetilde{a} \in \C}$ and~$\Big(e\left(\widetilde{\l}\right)\Big)_{\widetilde{\l} \in \textbf{E}}$ only satisfy the following relations:
$$\begin{array}{cclcclccl}
&&&z^{\widetilde{a}+\widetilde{b}}&=&z^{\widetilde{a}}z^{\widetilde{b}},&e\left(\widetilde{\l}_{1}+\widetilde{\l}_{2}\right)&=&
e\left(\widetilde{\l}_{1}\right)e\left(\widetilde{\l}_{2}\right),\\

&&&z^{\widetilde{a}}&=&z^{\widetilde{a}}\in \C((z)) \hbox{ for }\widetilde{a}\in \Z,&e(0)&=&1.
\end{array}~$$
\item The following rules of differentiation
$$\begin{array}{cclcclccl}
\d \log&=&1,&\d z^{\widetilde{a}}&=&\widetilde{a}z^{\widetilde{a}},&\d e\left(\widetilde{\l}\right)&=&\d \left(\widetilde{\l}\right) e\left(\widetilde{\l}\right),
\end{array}~$$
equip the ring with a differential structure, since these rules go to the quotient as can be readily checked.
\end{enumerate}
Proposition 3.22 in \cite{VdPS} tell us that the ring $\textbf{R}$ is an integral domain and its field of fractions $\textbf{K}$ has field of constants  equal to~$\C$. \par 

Let~$\widetilde{L}\in \mathrm{M}_{m}(\C)$, that is a complex $m\times m$ matrix.
Let~$\widetilde{L}=\widetilde{P}\left(\widetilde{D}+\widetilde{N}\right)\widetilde{P}^{-1}$, with~${\widetilde{D}=\mathrm{Diag}\left(\widetilde{d}_{i}\right)}$,~$\widetilde{d}_{i}\in \C$,~$\widetilde{N}$ nilpotent,~${\widetilde{D}\widetilde{N}=\widetilde{N}\widetilde{D}}$ and~$\widetilde{P} \in \mathrm{GL}_{m}(\C)$, that is a complex invertible $m\times m$ matrix, be the Jordan decomposition of the square matrix~$\widetilde{L}$. We construct the matrix 

$$z^{\widetilde{L}}:=\widetilde{P}\mathrm{Diag}\left(z^{\widetilde{d}_{i}}\right)e^{\widetilde{N}\log}\widetilde{P}^{-1}\in \mathrm{GL}_{m}\left( \textbf{K}\right).$$
One may check that it satisfies 
$$ \d z^{\widetilde{L}}=\widetilde{L}z^{\widetilde{L}}=z^{\widetilde{L}}\widetilde{L}.$$
Of course, if~$\widetilde{a}\in \C$ and~$\left(\widetilde{a}\right)\in \mathrm{M}_{1}(\C)$ is the corresponding matrix, we have~$ z^{\widetilde{a}}=z^{\left(\widetilde{a}\right)}$.\par 

Note that the intuitive interpretations of these symbols (resp. of the matrix $z^{\widetilde{L}}$) are~${\log=\log(z)}$,~$z^{\widetilde{a}}=e^{\widetilde{a}\log(z)}$ and~$e\left(\widetilde{\l}\right)=e^{\widetilde{\l}}$ (resp. is $e^{\widetilde{L}log(z)}$). Let~$\widetilde{f}$ be one these above functions (resp. an entry of the above matrix). Then~$\widetilde{f}$ has a natural interpretation as an analytic function on~$ \widetilde{\C}$, where~$\widetilde{\C}$ is the Riemann surface of the logarithm. We will use the analytic function instead of the symbol when we will consider asymptotic solutions in~$\S \ref{sec12}$. For the time being, however, we see them only as symbols.\\ \par

Let us consider~(\ref{eq13}). The Hukuhara-Turrittin theorem (see Theorem 3.1 in \cite{VdPS} for a statement that is trivially equivalent to the following) says that there exists a fundamental solution for (\ref{eq13}) of the form (for $\k\in \N^{*}$, $\mathrm{Id}_{\k}$ denotes the identity matrix of size~$\k$)
$$
\begin{array}{ll}
&\widetilde{H}(z)\mathrm{Diag}\left(z^{\widetilde{L}_{1}}e\left(\widetilde{\l}_{1}\right)\times\mathrm{Id}_{m_{1}},\dots,z^{\widetilde{L}_{r}}e\left(\widetilde{\l}_{r}\right)\times\mathrm{Id}_{m_{r}}\right)\\
:=&\widetilde{H}(z)\begin{pmatrix}
z^{\widetilde{L}_{1}}e\left(\widetilde{\l}_{1}\right)\times\mathrm{Id}_{m_{1}}&&\\
&\ddots&\\
&&z^{\widetilde{L}_{r}}e\left(\widetilde{\l}_{r}\right)\times\mathrm{Id}_{m_{r}}
\end{pmatrix},
\end{array}$$ 
where 
\begin{itemize}
\item~$\widetilde{H}\in \mathrm{GL}_{m}\Big(\C\left(\left(z\right)\right)\Big)$, 
\item~$\widetilde{L}_{j}\in \mathrm{M}_{m_{j}}(\C)$, for $1\leq j\leq r$ with $\sum m_{j}=m$,
\item~$\widetilde{\l}_{j}\in  \textbf{E}$, for $1\leq j\leq r$.
\end{itemize}

Note that the Hukuhara-Turrittin theorem works also for an arbitrary linear differential system in coefficients in $\C((z))$ with integral slopes. See \cite{VdPS}, $\S 3.3$, for a precise definition.

\subsection{Local analytic study of linear differential equations}\label{sec12}
Let~$\C\{z\}$ be the ring of germs of analytic functions in the neighbourhood of~$z=0$, and~$\C(\{z\})$ be its fraction field, that is the field of germs of meromorphic functions in the neighbourhood of~$z=0$. 
If~$a,b\in \R$ with~$a<b$, we define~$\mathcal{A}(a,b)$
as the ring of functions that are analytic in some punctured neighbourhood of~$0$ in $$ \overline{S}(a,b):=\left\{z\in \widetilde{\C} \Big|  \arg(z)\in ]a,b[\right\}.$$ 
In this subsection, we assume that  (\ref{eq10}) is a linear differential equation having coefficients in $\C(\{z\})$. We are interested in the existence of a basis of solutions of (\ref{eq10}), that belongs to~$\mathcal{A}(a,b)$, for some~$a<b$. \par We define the family of continuous map~$\left(\r_{a}\right)_{a\in \C}$, from the Riemann surface of the logarithm to itself, that sends~$z$ to~$e^{a\log (z)}$. One has~$\r_{b}\circ \r_{c}=\r_{bc}$ for any~$b,c\in \C$. For~$f\in \mathcal{A}(a,b)$ and~$c\in \C$, we define~$\rho_{c}\left(f\right):=f(z^{c})$. \par

\pagebreak[3]
\begin{defi}\label{defi1}
\begin{trivlist}
\item (1) Let~$k\in \N^{*}$. We define the formal Borel transform of order $k$,~$\hat{\mathcal{B}}_{k}$ as follows:
$$
\begin{array}{llll}
\hat{\mathcal{B}}_{k}:&\C[[z]]&\longrightarrow&\C[[\z]]\\
&\displaystyle\sum_{\ell\in \N} a_{\ell}z^{\ell}&\longmapsto&\displaystyle\sum_{\ell\in \N} \frac{a_{\ell}}{\G\left(1+\frac{\ell}{k}\right)}\z^{\ell},
\end{array}
$$
where~$\G$ is the Gamma function.
\item (2) Let~$d\in \R$ and~$k\in \N^{*}$.
Let~$f$ be a function such that there exists~$\e>0$, such that~${f\in \mathcal{A}(d-\e,d+\e)}$. We say that~$f$ belongs to~$\widetilde{\mathbb{H}}_{k}^{d}$, if~$f$ admits an analytic continuation defined on~$\overline{S}(d-\e,d+\e)$ that we will still call~$f$, with exponential growth of order~$k$ at infinity. This means that there exist constants~$J,L>0$, such that for~$\z\in\overline{S}(d-\e,d+\e)$:
$$|f(\z)|<J\exp\left(L|\z|^{k}\right).$$
\item (3)
 Let~$d\in \R$ and~$k\in \N^{*}$.  We define the Laplace transformations of order~$1$ and $k$ in the direction~$d$  as follow (see \cite{B}, Page 13, for a justification that the maps are defined)
$$
\begin{array}{llll}
\mathcal{L}_{1}^{d}:&\widetilde{\mathbb{H}}_{1}^{d}&\longrightarrow&\mathcal{A}\left(d-\frac{\pi}{2},d+\frac{\pi}{2}\right)\\
&f&\longmapsto& \displaystyle \int_{0}^{\infty e^{id}}z^{-1}f(\z)e^{-\left(\frac{\z}{z}\right)}d\z ,\\\\
\mathcal{L}_{k}^{d}:&\widetilde{\mathbb{H}}_{k}^{d}&\longrightarrow&\mathcal{A}\left(d-\frac{\pi}{2k},d+\frac{\pi}{2k}\right)\\
&g&\longmapsto& \r_{k}\circ  \mathcal{L}_{1}^{d}\circ \r_{1/k}\left(g\right).
\end{array}
$$
\item (4) Let~$d\in \R$,~$k\in \N^{*}$ and $\widetilde{h}\in \C[[z]]$. We say that $\widetilde{h}\in \widetilde{\mathbb{S}}_{k}^{d}$ if $\hat{\mathcal{B}}_{k}\left(\widetilde{h} \right)\in \C\{z\}\cap \widetilde{\mathbb{H}}_{k}^{d}$. In this case, we set 
$$\widetilde{S}^{d}\left(\widetilde{h}\right):=\mathcal{L}_{k}^{d}\circ\hat{\mathcal{B}}_{k}\left(\widetilde{h}\right) \in \mathcal{A}\left(d-\dfrac{\pi}{2k},d+\dfrac{\pi}{2k}\right).$$

\item (5) Let~$d\in \R$ and $\widetilde{h}\in \C((z))$. We say that $\widetilde{h}\in \widetilde{\mathbb{MS}}^{d}$ if there exist $k_{1},\dots,k_{r}\in \N^{*}$, $N\in \N$ and $\widetilde{h}_{k_{1}}\in \widetilde{\mathbb{S}}_{k_{1}}^{d},\dots,\widetilde{h}_{k_{r}}\in\widetilde{\mathbb{S}}_{k_{r}}^{d} $ such that
$z^{N}\widetilde{h}=\widetilde{h}_{k_{1}}+\dots+\widetilde{h}_{k_{r}}$. In this case, we set 
$$\widetilde{S}^{d}\left(\widetilde{h}\right):=z^{-N}\widetilde{S}^{d}\left(\widetilde{h}_{k_{1}}\right)+\dots+
z^{-N}\widetilde{S}^{d}\left(\widetilde{h}_{k_{r}}\right)\in \mathcal{A}\left(d-\dfrac{\pi}{2\e},d+\dfrac{\pi}{2\e}\right),$$
where $\e:=\max(k_{1},\dots,k_{r})$.
\end{trivlist}
\end{defi}

We still consider (\ref{eq13}). Let $\widetilde{H}$ be the matrix given by Hukuhara-Turrittin theorem. Note that the next theorem works also for an arbitrary linear differential system in coefficients in $\C(\{z\})$ with integral slopes.
\pagebreak[3]
\begin{theo}[\cite{B}, Theorem 2, $\S 6.4$, and Theorem 1, $\S 7.2$]\label{theo4}
There exists~$\widetilde{\Sigma}\subset \R$, finite modulo~$2\pi\Z$, called the set of singular directions of (\ref{eq13}), such that for all $d\in \R\setminus \widetilde{\Sigma}$, the entries of $\left(\widetilde{H}_{j,k}\right):=\widetilde{H}$ and $\widetilde{f}_{1},\dots,\widetilde{f}_{m}$ belong to $\widetilde{\mathbb{MS}}^{d}$. Let $\widetilde{S}^{d}\left(\widetilde{H}\right):=\widetilde{S}^{d}\left(\widetilde{H}_{j,k}\right)$. Moreover, 
there exists $\e>0$ such that we get an analytic fundamental solution  $$\widetilde{S}^{d}\left(\widetilde{H}\right)\mathrm{Diag} \left(e^{\widetilde{L}_{1} \log(z)}e^{\widetilde{\l}_{1}\times\mathrm{Id}_{m_{1}}},\dots,e^{\widetilde{L}_{r} \log(z)}e^{\widetilde{\l}_{r}\times\mathrm{Id}_{m_{r}}}\right)\in \mathrm{GL}_{m}\left(\mathcal{A}\left(d-\dfrac{\pi}{2\e},d+\dfrac{\pi}{2\e}\right)\right),$$
 for the linear differential system having coefficients in $\mathcal{A}\left(d-\dfrac{\pi}{2\e},d+\dfrac{\pi}{2\e}\right)$
 \begin{equation}\label{eq7}\d \widetilde{Y}=\widetilde{S}^{d}\left(\widetilde{C}\right)\widetilde{Y}, \hbox{ with }\widetilde{S}^{d}\left(\widetilde{C}\right):=\begin{pmatrix}
\widetilde{S}^{d}\left(\widetilde{f}_{1}\right)&1&&0\\
&\ddots&\ddots&\\
&&\ddots&1\\
0&&&\widetilde{S}^{d}\left(\widetilde{f}_{m}\right)
\end{pmatrix}.\end{equation}
\end{theo}

Note that for all $d\in \R\setminus \widetilde{\Sigma}$, we have also $$\widetilde{P}=\left(\d -\widetilde{f}_{m}\right)\dots\left(\d -\widetilde{f}_{1}\right)=\left(\d -\widetilde{S}^{d}\left(\widetilde{f}_{m}\right)\right)\dots\left(\d -\widetilde{S}^{d}\left(\widetilde{f}_{1}\right)\right),$$
which gives us an analytic basis of solutions of (\ref{eq10}), that belongs to the ring~$\mathcal{A}\left(d-\dfrac{\pi}{2\e},d+\dfrac{\pi}{2\e}\right)$.\\ \par 
 As a matter of fact, as we can see in Page $239$ of \cite{VdPS}, $$\widetilde{S}^{d}\left(\widetilde{H}\right)\mathrm{Diag} \left(e^{\widetilde{L}_{1} \log(z)}e^{\widetilde{\l}_{1}\times\mathrm{Id}_{m_{1}}},\dots,e^{\widetilde{L}_{r} \log(z)}e^{\widetilde{\l}_{r}\times\mathrm{Id}_{m_{r}}}\right)\in \mathrm{GL}_{m}\left(\mathcal{A}\left(d_{l}-\dfrac{\pi}{2\e},d_{l+1}+\frac{\pi}{2\e}\right)\right),$$ where the directions~${d_{l},d_{l+1}\in \widetilde{\Sigma}}$ are chosen such that~${\big]d_{l},d_{l+1}\big[ \bigcap \widetilde{\Sigma}=\varnothing}$. \\ \par 

Note that by definition, the analyticity holds on a subset of~$\widetilde{\C}$. 
Ramis has used the family of solutions $\left(\widetilde{S}^{d}\left(\widetilde{H}\right)\mathrm{Diag} \left(e^{\widetilde{L}_{1} \log(z)}e^{\widetilde{\l}_{1}\times\mathrm{Id}_{m_{1}}},\dots,e^{\widetilde{L}_{r} \log(z)}e^{\widetilde{\l}_{r}\times\mathrm{Id}_{m_{r}}}\right)\right)_{d\in \R\setminus \widetilde{\Sigma}}$ to build topological generators for the local differential Galois group of (\ref{eq13}). See Chapter 8 of \cite{VdPS} for more details. Remark also that the above family of solutions are involved in the local analytic classification of linear differential equations in coefficients in $\C(\{z\})$ with integral slopes. 

\pagebreak[3]
\section{Local study of linear $q$-difference equations}\label{sec2}
In this section, we make a short overview of the local formal and analytic classification of linear $q$-difference equations. See \cite{RSZ} for more details. 
\subsection{Local formal study of linear $q$-difference equations}\label{sec21}
Let~$q>1$ be fixed. Let us consider the monic linear $q$-difference equation in coefficients in $\C(\{z\})$
$$Q(y)=0.$$ 
As we can deduce from \cite{MZ}, $\S 3.1$, a $q$-difference operator may be factorized: there exist $g_{1},\dots,g_{m}\in \C(\{z^{1/\nu}\})\setminus \{0\}$, for some $\nu\in \N^{*}$, with for all $1\leq j<k\leq m$, $v_{0}\left(g_{j}\right)\leq v_{0}\left(g_{k}\right)$, such that 
$$Q=(\sq-g_{m})\dots (\sq-g_{1}).$$
The $q$-difference equation $Q(y)=0$
is equivalent to $P(y)=0$ with (remind that $\dq=\frac{\sq-\mathrm{Id}}{q-1}$)
$$P:=\left(\dq -f_{m}\right)\dots\left(\dq -f_{1}\right),$$
where $f_{1}:=\frac{g_{1}-1}{q-1},\dots,f_{m}:=\frac{g_{m}-1}{q-1}$. 
\begin{center}
\textbf{Until the end of the section we will assume that $g_{1},\dots,g_{m}\in \C(\{z\})\setminus \{0\}$.}\end{center}
 It follows that $P(y)=0$ is equivalent to 
\begin{equation}\label{eq15}
\sq Y=CY, \hbox{ where }
C:=\begin{pmatrix}
1+(q-1)f_{1} &q-1&&0\\
&\ddots&\ddots&\\
&&\ddots& q-1\\
0&&& 1+(q-1)f_{m}
\end{pmatrix}.
\end{equation}
\pagebreak[3]
\begin{rem}
The opposite of $v_{0}\left(g_{1}\right),\dots,v_{0}\left(g_{m}\right)$ equal to the slopes of $Q$. The multiplicity of the slope $-v_{0}\left(g_{j}\right)$ for $1\leq j\leq m$ corresponds to the number $m_{j}\in \N$, of $g_{i}$ having valuation $v_{0}\left(g_{j}\right)$. See \cite{RSZ}, $\S 2.2$, for a precise definition.
\end{rem}

 Let~$K_{0}$ be a sub-field of ${\C\left(\left(z\right)\right)}$, stable by~$\sq$. Let~$A,B\in \mathrm{GL}_{m}\left(K_{0}\right)$. The two~$q$-difference systems,~$\sq Y=AY$ and~$\sq Y=BY$ are equivalent over~$K_{0}$, if there exists~${P\in \mathrm{GL}_{m}(K_{0})}$, called gauge transformation, such that
$$A=P[B]_{\displaystyle\sq}:=(\sq P)BP^{-1}.$$ In particular,
$$\sq Y=BY\Longleftrightarrow\sq \left(PY\right)=APY.$$ \par
Conversely, if there exist $A,B,P\in \mathrm{GL}_{m}(K_{0})$ such that
$\sq Y=BY$, $ \sq Z=AZ$ and $Z=PY$, then 
$$A=P\left[ B\right]_{\displaystyle\sq}.$$
Note that the next theorem works also for an arbitrary linear $q$-difference system in coefficients in $\C((z))$ with integral slopes. From \cite{RSZ},~$\S 2.2$, we may deduce:
\pagebreak[3]
\begin{theo}\label{theo6}
Let $C$ be defined in (\ref{eq15}).
 We have existence and uniqueness of 
\begin{itemize}
\item $B_{j}\in \mathrm{GL}_{m_{j}}(\C)$, matrices in the Jordan normal form with $\sum m_{j}=m$,
\item $\mu_{1}<\dots <\mu_{r}$ elements of $\Z$,
\item $\hat{G}\in \mathrm{GL}_{m}\Big(\C\left(\left(z\right)\right)\Big)$, with diagonal entries that have $z$-valuation $0$ and that have constant terms equal to $1$, 
\end{itemize} such that: 
$$C=\hat{G}\left[D\right]_{\displaystyle\sq},\hbox{ where }D:=\mathrm{Diag} \left(B_{1}z^{\mu_{1}},\dots,B_{r}z^{\mu_{r}}\right).$$
\end{theo}

\pagebreak[3]
\begin{rem}
The opposite of the $\mu_{j}$ are the slopes of the Newton polygon of (\ref{eq15}) and the $m_{j}$ are the corresponding multiplicities. See \cite{RSZ}, $\S 2.2$, for a precise definition. In particular, the above theorem gives the local formal classification of linear $q$-difference systems with integral slopes. See \cite{vdPR} for the local formal classification of linear $q$-difference systems in the general case.
\end{rem}

\pagebreak[3]
\subsection{Local analytic study of linear $q$-difference equations}\label{sec22}
\pagebreak[3]
\begin{defi}
We say that $T\in \mathrm{GL}_{m}(\C(z))$ is a Birkhoff-Guenther matrix, if there exist $\mu_{1}<\dots<\mu_{r}$ elements of $\Z$, $m_{1},\dots,m_{r}$ positive integers which sum is $m$, $B_{j}\in \mathrm{GL}_{m_{j}}(\C)$, $U_{j,k}$, $m_{j}$ times $m_{k}$ matrices with coefficients in $\displaystyle \sum_{\nu=\mu_{j}}^{\mu_{k}-1}\C z^{\nu}$, such that 
$$T=\begin{pmatrix}
z^{\mu_{1}}B_{1}&\dots&\dots&\dots&\dots\\
&\ddots&\dots&U_{j,k}&\dots \\
&&\ddots&\dots&\dots \\
&&&\ddots&\dots \\
0&&&&z^{\mu_{r}}B_{r}
\end{pmatrix}. $$
\end{defi}
 The next theorem says that after an analytic gauge transformation, we may put (\ref{eq15}) in the Birkhoff-Guenther normal form. Note that the result is true for any linear $q$-difference system in coefficients in $\C(\{z\})$ having integral slopes. This result is used in \cite{RSZ} to make the local analytic classification of such $q$-difference systems. From \cite{RSZ}, $\S 3.3.2$, we may deduce:

\pagebreak[3]
\begin{theo}\label{theo1}
Let $C$ be defined in (\ref{eq15}). Let~$\mu_{1}<\dots<\mu_{r}$, ~$m_{1}\dots,m_{r}$ and $B_{1},\dots,B_{r}$ be defined as in Theorem \ref{theo6}. We have the existence of a unique pair $(F,T)$, where  $F\in\mathrm{GL}_{m}\Big(\C(\{z\})\Big)$,  $T$ is a Birkhoff-Guenther matrix of the form
$$T:=\begin{pmatrix}
z^{\mu_{1}}B_{1}&&U_{j,k}\\
&\ddots& \\
0&&z^{\mu_{r}}B_{r}
\end{pmatrix}, $$
such that $$C=F[T]_{\displaystyle\sq}.$$
\end{theo}

Let us introduce some notations. Let $\mathcal{M}(\C^{*})$ (resp. $\mathcal{M}(\C^{*},0)$) be the field of meromorphic functions on $\C^{*}$ (resp. meromorphic functions on some punctured neighbourhood of~$0$ in~$\C^{*}$). Let $K_{0}$ be a sub-field of $\mathcal{M}(\C^{*},0)$ stable by $\sq$. Let~$A,B\in \mathrm{GL}_{m}\left(K_{0}\right)$. The two~$q$-difference systems,~$\sq Y=AY$ and~$\sq Y=BY$ are equivalent over~$K_{0}$, if there exists~$P\in \mathrm{GL}_{m}(K_{0})$, called gauge transformation, such that
$$A=P[B]_{\displaystyle\sq}:=(\sq P)BP^{-1}.$$ 
Let us define $\Sigma \subset \C^{*}$, finite modulo $q^{\Z}$, as follows (see Theorem \ref{theo6} for the definition of the matrices $B_{1},\dots,B_{r}$ and the integers $\mu_{1},\dots,\mu_{r}$):
$$\Sigma:=\displaystyle \bigcup_{1\leq j<k\leq r} S_{j,k}, \hbox{ where } S_{j,k}:=\left\{-a\in \C^{*}\Big|q^{\Z}a^{-\mu_{j}}\mathrm{Sp}(B_{j})\cap  q^{\Z}a^{-\mu_{k}}\mathrm{Sp}(B_{k})\neq \varnothing \right\},$$
and $\mathrm{Sp}$ denotes the spectrum.

\pagebreak[3]
\begin{theo}[\cite{RSZ}, Theorem 6.1.2]
Let $D,T$ be the matrices defined in Theorems~\ref{theo6} and \ref{theo1}.  For all ${\l\in \C^{*}\setminus \Sigma}$, there exists an unique matrix $$\hat{H}^{\left[\l\right]}:=\begin{pmatrix}
\mathrm {Id}_{m_{1}}&&\hat{H}^{\left[\l\right]}_{j,k}\\
&\ddots&\\
0&&\mathrm {Id}_{m_{r}}
\end{pmatrix}\in \mathrm{GL}_{m}\Big(\mathcal{M}(\C^{*})\Big),$$ solution of 
$T=\hat{H}^{\left[\l\right]}[D]_{\displaystyle\sq}$, such that for all $1\leq j<k\leq r$, the germs of the entries of $\hat{H}^{\left[\l\right]}_{j,k}$ at $0$ have all their poles on $-\l q^{\Z}$ with multiplicities at most  $\mu_{k}-\mu_{j}$. 
\end{theo}
In $\S 3.3.3$ of \cite{RSZ}, it is shown the existence and the uniqueness of $$\hat{H}:=\begin{pmatrix}
\mathrm {Id}_{m_{1}}&&\hat{H}_{j,k}\\
&\ddots& \\
0&&\mathrm {Id}_{m_{r}}
\end{pmatrix}\in \mathrm{GL}_{m}\Big(\C[[z]]\Big),$$ formal gauge transformation, that satisfies
$$T=\hat{H}[D]_{\displaystyle\sq}.$$ 
The matrix $\hat{H}^{\left[\l\right]}$ is $q$-Gevrey asymptotic to  $\hat{H}:=\sum_{\ell\in\N} \hat{H}_{\ell}z^{\ell}$ along the divisor $-\l q^{\Z}$. This means that for all $W\subset \C^{*}$, open set that satisfies $\inf_{z\in W,\zeta\in -\l q^{\Z}}\left|1-\frac{z}{\zeta} \right|>0$,  there exists~$M >0$, such that for all~$N \in \N^{*}$ and all ${z\in W}$, ($\left\| .\right\|_{\infty}$ denotes the infinite norm)
$$\left\| \hat{H}^{\left[\l\right]}(z)-\sum_{\ell= 0}^{N-1} \hat{H}_{\ell}z^{\ell}\right\|_{\infty}\leq M^{N}q^{N^{2}/2\mu_{1}} |z|^{N}.$$   
Note that if $z^{N}F$ (see Theorem~\ref{theo6} for the definition of $F$) with $N\in \N$ has entries in~$\C[[z]]$, then $z^{N}F\hat{H}^{\left[\l\right]}$ is $q$-Gevrey asymptotic to $z^{N}F\hat{H}$ along the divisor $-\l q^{\Z}$.\\ \par 
We still consider (\ref{eq15}).
Let $\Lambda\in \mathrm{GL}_{m}\left(\mathcal{M}(\C^{*})\right)$ (see below for the existence of $\L$) be any solution of 
\begin{equation}\label{eq2}
\sq \Lambda=\mathrm{Diag}\left(B_{1}z^{\mu_{1}},\dots,B_{r}z^{\mu_{r}}\right)\Lambda=
\Lambda \mathrm{Diag}\left(B_{1}z^{\mu_{1}},\dots,B_{r}z^{\mu_{r}}\right).
\end{equation}
From what precede, we obtain that for every $\l\in \C^{*}\setminus \Sigma$, we get a meromorphic fundamental solution for (\ref{eq15})
$$U_{\Lambda}^{[\l]}:=F\hat{H}^{\left[ \l\right]}\Lambda\in \mathrm{GL}_{m}(\mathcal{M}(\C^{*},0)).$$
 We finish the subsection by giving an explicit meromorphic solution of (\ref{eq2}), but before, we need to introduce some notations.
 For~$a\in \C^{*}$, let us consider, $\T_{q}(z):=\displaystyle \sum_{\ell \in \Z} q^{\frac{-\ell(\ell+1)}{2}}z^{\ell}$,~$l_{q}(z):=\dfrac{\d\left(\T_{q}(z)\right)} {\T_{q}(z)}$, ~${\L_{q,a}(z):=\frac{\T_{q}(z)}{\T_{q}(z/a)}}$.
They satisfy the~$q$-difference equations: \\
\begin{itemize}
\item $\sq \T_{q}=z\T_{q}$.\\
\item~$\sq l_{q}=l_{q} +1$.\\
\item~$\sq \L_{q,a}=a\L_{q,a}$.\\
\end{itemize}
Let~$A$ be an invertible matrix with complex coefficients and consider now the decomposition in Jordan normal form~$A=P(D'U)P^{-1}$, where~$D'=\mathrm{Diag}(d_{i})$ is diagonal,~$U$ is a unipotent upper triangular matrix with~$D'U=UD'$, and~$P$ is an invertible matrix with complex coefficients. Following \cite{S00}, we construct the matrix:
$$
\L_{q,A}:=P\left(\mathrm{Diag}\left(\L_{q,d_{i}}\right)e^{\log(U)l_{q}}\right)P^{-1}\in \mathrm{GL}_{m}\Big(\C\left( l_{q},\left(\L_{q,a}\right)_{a \in \C^{*}}\right)\Big)$$
that satisfies:
$$
\sq \L_{q,A}=A\L_{q,A}=\L_{q,A}A.$$
Let~$a\in \C^{*}$ and consider the corresponding matrix~$(a)\in \mathrm{GL}_{1}(\C)$. By construction, we have~$\L_{q,a}=\L_{q,(a)}$.\\ \par

Then, the following matrix is solution of (\ref{eq2})
$$\Lambda_{0}:=\mathrm{Diag}\left(\L_{q,B_{1}}\T_{q}\left(z\right)^{\mu_{1}}\times \mathrm{Id}_{m_{1}},\dots,\L_{q,B_{r}}\T_{q}\left(z\right)^{\mu_{r}}\times \mathrm{Id}_{m_{r}} \right)\in \mathrm{GL}_{m}\left(\mathcal{M}(\C^{*})\right).$$

Note that the family of solutions $\left(U_{\Lambda_{0}}^{[\l]}\right)_{\l\in \C^{*}\setminus \Sigma}$ are involved to build topological generator for the local Galois group of (\ref{eq15}). See \cite{RS07,RS09}. \\ \par 

\pagebreak[3]
\section{Integral representation of the meromorphic solutions}\label{sec3}
Sauloy's algebraic summation, see \cite{S04b}, gives an explicit way to compute the meromorphic solutions of  $\S\ref{sec22}$. Unfortunately, the behaviour of the expression of the solutions when $q$ goes to $1$ seems complicated to understand. Under convenient assumptions, we are going to give in this section, another expression of the meromorphic solutions. \\ \par

For $g=\sum_{\ell=n}^{\infty} g_{\ell}z^{\ell}\in \C((z))\setminus \{0\}$ with $g_{n}\neq 0$, let  $t_{0}(g):=g_{n}$. Throughout the paper, we will make the convention that $t_{0}(0):=0$. We remind that $v_{0}$ denotes the $z$-valuation. Let us consider ${f_{1},\dots,f_{m}\in  \C\left(\left\{z\right\}\right)}$ of (\ref{eq15}) and assume that (non resonance condition)
\begin{equation}\label{eq20}
v_{0}\left(1+(q-1)f_{j}\right)=v_{0}\left(1+(q-1)f_{k}\right)\Longrightarrow t_{0}(1+(q-1)f_{j})\notin q^{\Z}t_{0}\left(1+(q-1)f_{k}\right).
\end{equation}
Note that if the Newton polygon of (\ref{eq15}) has only positive slopes with multiplicity one, then the non resonance condition is satisfied.
The next lemma gives the expression of the bloc diagonal matrix appearing in Theorem~\ref{theo6}. 

\pagebreak[3]
\begin{lem}\label{lem1}
There exists a unique ${(\hat{g}_{j,k}):=\hat{G}\in \mathrm{GL}_{m}\Big(\C((z))\Big)}$, upper triangular with for all ${1\leq j\leq m}$, $v_{0}(\hat{g}_{j,j})=0$, $t_{0}\left(\hat{g}_{j,j}\right)=1$, such that 
$$C=\hat{G}\left[\mathrm{Diag}\left(z^{v_{0}(1+(q-1)f_{1})}t_{0}\left(1+(q-1)f_{1}\right),\dots,z^{v_{0}(1+(q-1)f_{m})}t_{0}\left(1+(q-1)f_{m}\right)\right)\right]_{\displaystyle\sq}.$$
\end{lem}

\begin{proof}
Due to the assumptions we have made on $f_{1},\dots,f_{m}$, For all $1\leq j\leq m$ (resp. ${1\leq j<k\leq m}$) there exists $\hat{g}_{j,j}\in \C[[z]]$ with constant term equal to $1$ (resp.~${\hat{g}_{j,k}\in \C((z))}$), solution of 
$$z^{v_{0}(1+(q-1)f_{j})}t_{0}\left(1+(q-1)f_{j}\right)\sq \hat{g}_{j,j}=(1+(q-1)f_{j})\hat{g}_{j,j}$$
resp.
\begin{equation}\label{eq12}
z^{v_{0}(1+(q-1)f_{k})}t_{0}\left(1+(q-1)f_{k}\right)\sq \hat{g}_{j,k}=(1+(q-1)f_{j})\hat{g}_{j,k}+(q-1)\hat{g}_{j+1,k}.
\end{equation}
For $1\leq k<j\leq m$, put $\hat{g}_{j,k}:=0$. Then $(\hat{g}_{j,k}):=\hat{G}$ is the unique matrix that satisfies the required properties.
\end{proof}

\pagebreak[3]
\begin{rem}\label{rem3}
Note that the matrices $B_{j}$ of Theorem \ref{theo6} are all diagonal.
\end{rem}
\pagebreak[3]
\begin{rem}\label{rem5}
Using (\ref{eq12}), we obtain that if for all $1\leq j\leq m$, $v_{0}(1+(q-1)f_{j})\leq 0$, then $z\mapsto\hat{G}(z,q)\in \mathrm{GL}_{m}\Big(\C[[z]]\Big)$. More precisely, let $1\leq j<k\leq m$. We have 
$$v_{0}\left(\hat{g}_{j,k}\right)=-\displaystyle \sum_{\ell=j}^{k-1}v_{0}(1+(q-1)f_{\ell}).$$
\end{rem}
Let us define $\Sigma\subset \C^{*}$ as in $\S\ref{sec22}$ and let $\L\in \mathrm{GL}_{m}\left(\mathcal{M}(\C^{*})\right)$ be any diagonal matrix solution of 
$$
\begin{array}{lll}
\sq \Lambda&=&\mathrm{Diag}\left(
z^{v_{0}(1+(q-1)f_{1})}t_{0}\left(1+(q-1)f_{1}\right),\dots,z^{v_{0}(1+(q-1)f_{m})}t_{0}\left(1+(q-1)f_{m}\right)\right)\Lambda\\
&=&\Lambda\mathrm{Diag}\left(
z^{v_{0}(1+(q-1)f_{1})}t_{0}\left(1+(q-1)f_{1}\right),\dots,z^{v_{0}(1+(q-1)f_{m})}t_{0}\left(1+(q-1)f_{m}\right)\right).\end{array}$$ 

Using Lemma \ref{lem1}, we find that for every $\l\in \C^{*}\setminus \Sigma$, the fundamental solution for (\ref{eq15}) defined in $\S \ref{sec22}$ is of the form $$\left(u_{\Lambda,j,k}^{[\l]}\right):=U_{\Lambda}^{[\l]}=F\hat{H}^{\left[ \l\right]}\Lambda\in \mathrm{GL}_{m}(\mathcal{M}(\C^{*},0)). $$

 We will need a $q$-discrete analogue of the integration. For $N\in  \N^{*}\cup \{+\infty\}$, $z\in \C^{*}$, let us set the Jackson integral
																
$$\displaystyle \int_{q^{-N}z}^{z} f(t)d_{q}t:=(q-1)\displaystyle \sum_{\ell=-N}^{-1}f\left(q^{\ell}z\right)q^{\ell}z,$$
whenever the right hand side converges.
 Roughly speaking, Jackson integral degenerates into classical integral when $q$ goes to $1$, which means that for a convenient choice of function~$f$, we have on a convenient domain $$\displaystyle \int_{0}^{z} f(t)d_{q}t \underset{q \to 1}{\longrightarrow} \displaystyle \int_{0}^{z} f(t)dt.$$ 
Note also that for a convenient $f$, we have $$\dq \left(\int_{0}^{z} f(t)\frac{d_{q}t}{t}\right)=f(z).$$

\pagebreak[3]
\begin{ex}
Let $m=2$, $f_{1}:=\frac{-q^{-1}z^{-1}}{1+(q-1)q^{-1}z^{-1}}$ and $f_{2}:=0$. Then, for some diagonal matrix $\Lambda$, for some convenient $\l\in \C^{*}$, $z\in \C^{*}$, we have
$$U_{\Lambda}^{[\l]}(z):=\begin{pmatrix}
e_{q}(z^{-1})&u_{\Lambda,1,2}^{[\l]}(z)\\
0&1
\end{pmatrix}. $$
We want to give an integral formula for the entry $u_{\Lambda,1,2}^{[\l]}(z)$. Let us fix $\l\in \C^{*}$, $z\in \C^{*}$ such that for all $N\in \N$, $u_{\Lambda,1,2}^{[\l]}$ is analytic at $q^{-N}z$. Iterating the linear $q$-difference equation $$\sq \left(u_{\Lambda,1,2}^{[\l]}\right)=(1+(q-1)f_{1})u_{\Lambda,1,2}^{[\l]}+(q-1),$$  we find that for all $N\in \N^{*}$, $u_{\Lambda,1,2}^{[\l]}(z)$ equals to,
$$u_{\Lambda,1,2}^{[\l]}\left(q^{-N}z\right)\displaystyle\prod_{\ell=-N}^{-1}\left(1+(q-1)f_{1}\left(q^{\ell}z\right)\right)+\displaystyle\sum_{\nu=-N}^{-1}(q-1)\prod_{\ell=\nu+1}^{-1}\left(1+(q-1)f_{1}\left(q^{\ell}z\right)\right). $$
Using additionally the $q$-difference equation $$\sq \left(e_{q}(z^{-1})\right)=(1+(q-1)f_{1})e_{q}(z^{-1}),$$
we find that,
$$u_{\Lambda,1,2}^{[\l]}(z)=u_{\Lambda,1,2}^{[\l]}\left(q^{-N}z\right)\frac{e_{q}(z^{-1})}{e_{q}(q^{N}z^{-1})}+e_{q}(z^{-1})\displaystyle\int_{q^{-N}z}^{z} \frac{1}{e_{q}(q^{-1}t^{-1})}\frac{d_{q}t}{t}. $$
This solution looks like a solution obtained via a $q$-deformation of the variation of constants method. Indeed, using this method, we find a solution of $z\d \widetilde{y}-\widetilde{y}=z$ given by
$$\exp(z^{-1})\displaystyle\int_{0}^{z} \exp(-t^{-1})\frac{dt}{t}.$$
The goal of the next proposition is to generalise this example.
\end{ex}

\pagebreak[3]
\begin{propo}\label{propo1}
Let $\l\in \C^{*}\setminus\Sigma$. For all $1\leq j<k\leq m$, $N\in \N^{*}$, and $z\in \C^{*}\setminus -\l q^{\Z}$ sufficiently close to $0$, we have the equality of functions

$$u_{\Lambda,j,k}^{[\l]}(z)=u_{\Lambda,j,k}^{[\l]}\left(q^{-N}z\right)\frac{u_{\Lambda,j,j}^{[\l]}(z)}{u_{\Lambda,j,j}^{[\l]}\left(q^{-N}z\right)}+u_{\Lambda,j,j}^{[\l]}(z)\displaystyle\int_{q^{-N}z}^{z} \frac{u_{\Lambda,j+1,k}^{[\l]}(t)}{u_{\Lambda,j,j}^{[\l]}(qt)}\frac{d_{q}t}{t}.$$
Let us fix $z_{0}\in \C^{*}\setminus -\l q^{\Z}$ such that for all integers $1\leq j\leq m$, the function $u_{\Lambda,j,j}^{[\l]}$ is analytic on ${\{q^{\ell}z_{0},\ell \leq 0\}}$. If we assume that the following limit exists $\lim\limits_{N \to +\infty}\dfrac{t_{0}\left(\hat{g}_{j,k}\right) (q^{-N}z_{0})^{v_{0}(\hat{g}_{j,k})}u_{\Lambda,k,k}^{[\l]}(q^{-N}z_{0})}{u_{\Lambda,j,j}^{[\l]}(q^{-N}z_{0})}=:c_{\Lambda,j,k}^{[\l]}(z_{0})$, then we obtain that the Jackson integral $\displaystyle\int_{0}^{z_{0}} \frac{u_{\Lambda,j+1,k}^{[\l]}(t)}{u_{\Lambda,j,j}^{[\l]}(qt)}\frac{d_{q}t}{t}$ is well defined and that we have
$$u_{\Lambda,j,k}^{[\l]}(z_{0})=c_{\Lambda,j,k}^{[\l]}(z_{0})u_{\Lambda,j,j}^{[\l]}(z_{0})+u_{\Lambda,j,j}^{[\l]}(z_{0})\displaystyle\int_{0}^{z_{0}} \frac{u_{\Lambda,j+1,k}^{[\l]}(t)}{u_{\Lambda,j,j}^{[\l]}(qt)}\frac{d_{q}t}{t}.$$
\end{propo}

\begin{proof}
Iterating the linear $q$-difference equation $$\sq \left(u_{\Lambda,j,k}^{[\l]}\right)=(1+(q-1)f_{j})u_{\Lambda,j,k}^{[\l]}+(q-1)u_{\Lambda,j+1,k}^{[\l]},$$  we find for all $N\in \N^{*}$, $u_{\Lambda,j,k}^{[\l]}(z)$ equals (we remind that the empty product has value $1$)
$$u_{\Lambda,j,k}^{[\l]}\left(q^{-N}z\right)\displaystyle\prod_{\ell=-N}^{-1}\left(1+(q-1)f_{j}\left(q^{N}z\right)\right)+\displaystyle\sum_{\nu=-N}^{-1}(q-1)u_{\Lambda,j+1,k}^{[\l]}(q^{\nu}z)\prod_{\ell=\nu+1}^{-1}\left(1+(q-1)f_{j}\left(q^{\ell}z\right)\right). $$
Using additionally the $q$-difference equation $$\sq \left(u_{\Lambda,j,j}^{[\l]}\right)=(1+(q-1)f_{j})u_{\Lambda,j,j}^{[\l]},$$
we find that 
$$u_{\Lambda,j,k}^{[\l]}(z)=u_{\Lambda,j,k}^{[\l]}\left(q^{-N}z\right)\frac{u_{\Lambda,j,j}^{[\l]}\left(z\right)}{u_{\Lambda,j,j}^{[\l]}\left(q^{-N}z\right)}+u_{\Lambda,j,j}^{[\l]}(z)\displaystyle\int_{q^{-N}z}^{z} \frac{u_{\Lambda,j+1,k}^{[\l]}(t)}{u_{\Lambda,j,j}^{[\l]}(qt)}\frac{d_{q}t}{t}. $$
Then, we obtain that  $u_{\Lambda,j,k}^{[\l]}(z)$ equals to

\begin{equation}\label{eq1}\begin{array}{ll}
&\dfrac{u_{\Lambda,j,k}^{[\l]}\left(q^{-N}z\right)}{t_{0}\left(\hat{g}_{j,k}\right) (q^{-N}z)^{v_{0}(\hat{g}_{j,k})}u_{\Lambda,k,k}^{[\l]}\left(q^{-N}z\right)}\dfrac{t_{0}\left(\hat{g}_{j,k}\right) (q^{-N}z)^{v_{0}(\hat{g}_{j,k})}u_{\Lambda,k,k}^{[\l]}\left(q^{-N}z\right)}{u_{\Lambda,j,j}^{[\l]}\left(q^{-N}z\right)}u_{\Lambda,j,j}^{[\l]}\left(z\right)\\
+&u_{\Lambda,j,j}^{[\l]}(z)\displaystyle\int_{q^{-N}z}^{z} \frac{u_{\Lambda,j+1,k}^{[\l]}(t)}{u_{\Lambda,j,j}^{[\l]}(qt)}\frac{d_{q}t}{t}.\end{array}\end{equation}

Let us fix $z_{0}\in \C^{*}\setminus -\l q^{\Z}$ such that for all integers $1\leq j\leq m$, $u_{\Lambda,j,j}^{[\l]}$ is analytic on ${\{q^{\ell}z_{0},\ell \leq 0\}}$ and assume now that the following limit exists 

\begin{equation}\label{eq3}
\lim\limits_{N \to +\infty}\frac{t_{0}\left(\hat{g}_{j,k}\right) (q^{-N}z_{0})^{v_{0}(\hat{g}_{j,k})}u_{\Lambda,k,k}^{[\l]}(q^{-N}z_{0})}{u_{\Lambda,j,j}^{[\l]}(q^{-N}z_{0})}=:c_{\Lambda,j,k}^{[\l]}(z_{0}).\end{equation}

We remind that $v_{0}\left(\hat{g}_{k,k}\right)=0$. Then, the function $\dfrac{u_{\Lambda,j,k}^{[\l]}}{t_{0}\left(\hat{g}_{j,k}\right) z^{v_{0}(\hat{g}_{j,k})}u_{\Lambda,k,k}^{[\l]}}$ is $q$-Gevrey asymptotic to the series $\frac{\hat{g}_{j,k}}{t_{0}\left(\hat{g}_{j,k}\right) z^{v_{0}(\hat{g}_{j,k})}\hat{g}_{k,k}}\in \C[[z]] $ along the divisor $-\l q^{\Z}$. Since $\hat{g}_{k,k}$ has constant term equal to $1$, we find that  $$\lim\limits_{N \to +\infty}\frac{u_{\Lambda,j,k}^{[\l]}(q^{-N}z_{0})}{t_{0}\left(\hat{g}_{j,k}\right) (q^{-N}z_{0})^{v_{0}(\hat{g}_{j,k})}u_{\Lambda,k,k}^{[\l]}(q^{-N}z_{0})}=1.$$
Then, we use this limit, (\ref{eq1}) and (\ref{eq3}), to deduce that the Jackson integral $\displaystyle\int_{0}^{z_{0}} \frac{u_{\Lambda,j+1,k}^{[\l]}(t)}{u_{\Lambda,j,j}^{[\l]}(qt)}\frac{d_{q}t}{t}$ is well defined and that we have 
$$u_{\Lambda,j,k}^{[\l]}(z_{0})=c_{\Lambda,j,k}^{[\l]}(z_{0})u_{\Lambda,j,j}^{[\l]}(z_{0})+u_{\Lambda,j,j}^{[\l]}(z_{0})\displaystyle\int_{0}^{z_{0}} \frac{u_{\Lambda,j+1,k}^{[\l]}(t)}{u_{\Lambda,j,j}^{[\l]}(qt)}\frac{d_{q}t}{t}.$$
\end{proof}

\pagebreak[3]
\section{$q$-dependency of meromorphic solutions of linear differential equations.}\label{sec4}
From now, we see~$q$ as a parameter in~$]1,\infty[$. When we say that~$q$ is close to~$1$, we mean that~$q$ will be in the neighbourhood of~$1$ in~$]1,\infty[$. Formally, we have the convergence~$\lim\limits_{q \to 1}\dq =\d$. In $\S \ref{sec41}$ we state and prove a preliminary result about confluence. Given a linear differential equation $\widetilde{\D}$ and a family of linear $q$-difference equations $\D_{q}$, we state that under convenient assumptions, a family of basis of meromorphic solutions of $\D_{q}$ defined in $\S \ref{sec22}$ converges, when $q$ goes to $1$, to a basis of meromorphic solutions of $\widetilde{\D}$. In $\S \ref{sec43}$, we use the result of $\S\ref{sec41}$ on an example using an approach that will be generalized in $\S \ref{sec42}$, where we state and prove our main result. Under convenient assumptions, we prove that given a linear differential equation $\widetilde{\D}$ and a basis of meromorphic solutions, we may define a family of linear $q$-difference equations $\D_{q}$, that is a $q$-deformation of $\widetilde{\D}$, and for every $q$ close to~$1$, a basis of meromorphic solutions of $\D_{q}$, that is one of these defined in $\S \ref{sec22}$, and that converges, when $q$ goes to $1$, to the given basis of meromorphic solutions of $\widetilde{\D}$. We refer to $\S \ref{sec1},\S\ref{sec2},\S\ref{sec3}$ for the notations used in this section.

\pagebreak[3]
\subsection{Confluence of meromorphic solutions of linear differential equations.}\label{sec41}

Let us consider a family of equations (see below for the definition of the coefficients)
$$
\left\{\begin{array}{lll}
\D_{q}&:=&\left(\dq -f_{m}(z,q)\right)\dots\left(\dq -f_{1}(z,q)\right)\\
\widetilde{\D}&:=&\left(\d -\widetilde{f}_{m}(z)\right)\dots\left(\d -\widetilde{f}_{1}(z)\right).
\end{array} \right. 
$$
As in $\S \ref{sec1},\S \ref{sec2}$, we are going to see the equations as systems:
$$\left\{\begin{array}{lll}
\dq Y(z,q)&=&C(z,q)Y(z,q)\\ 
 \d \widetilde{Y}(z)&=&\widetilde{C}(z)\widetilde{Y}(z).
 \end{array} \right. $$
Let us assume that: \begin{trivlist}
\item \textbf{(H1)} For all $q$, $z\mapsto f_{1}(z,q)\in \C(\{z\}),\dots,z\mapsto f_{m}(z,q)\in \C(\{z\})$. Furthermore, assume that $\widetilde{\D}$ is a linear differential system in coefficients in $\C(\{z\})$ and $\widetilde{f}_{1},\dots,\widetilde{f}_{m}\in \C((z))$.
\item \textbf{(H2)} For all $1\leq j\leq m$,  $t_{0}\left(f_{j}\right)$ and $\arg\left(t_{0}\left(f_{j}\right)\right)$ converge  when $q\to 1$. Moreover, we assume that for $q$ sufficiently close to~$1$, $ v_{0}\left(1+(q-1)f_{j}\right)$ is constant, satisfies for all $1\leq j<k\leq m$, $v_{0}\left( 1+(q-1)f_{j}\right)\leq v_{0}\left( 1+(q-1)f_{k}\right)$ and the non resonance condition (\ref{eq20}) is satisfied.
\item \textbf{(H3)} For all $1\leq j<k\leq m$, $v_{0}\left( \widetilde{f}_{j}\right)\leq v_{0}\left( \widetilde{f}_{k}\right)$.
\end{trivlist}

Previous assumptions and Lemma \ref{lem1} imply the existence of a set of singular directions $\widetilde{\Sigma}\subset\R$,  finite modulo $2\pi$, such that for all $\l\in \C^{*}$ with $d:=\arg(\l) \in \R\setminus \widetilde{\Sigma}$, for all $q$ close to $1$, for every $z\mapsto \L(z,q)\in \mathrm{GL}_{m}\left(\mathcal{M}(\C^{*})\right)$, family of diagonal matrices solution of 

\begin{equation}\label{eq8}
\begin{array}{ll}
\sq \Lambda&=\mathrm{Diag}\left(
z^{v_{0}(1+(q-1)f_{1})}t_{0}\left(1+(q-1)f_{1}\right),\dots,z^{v_{0}(1+(q-1)f_{m})}t_{0}\left(1+(q-1)f_{m}\right)\right)\Lambda\\
&=\Lambda\mathrm{Diag}\left(
z^{v_{0}(1+(q-1)f_{1})}t_{0}\left(1+(q-1)f_{1}\right),\dots,z^{v_{0}(1+(q-1)f_{m})}t_{0}\left(1+(q-1)f_{m}\right)\right),\end{array}
\end{equation}
we may consider $\left(u_{\L,j,k}^{[\l]}(z,q)\right):=U_{\L}^{[\l]}(z,q)$, the fundamental solution for the $q$-difference system ${\dq Y(z,q)=C(z,q)Y(z,q)}$, which is defined in $\S \ref{sec22}$, and $\widetilde{S}^{d}\left(\widetilde{H}\right)\mathrm{Diag} \left(e^{\widetilde{L}_{1} \log(z)}e^{\widetilde{\l}_{1}\times\mathrm{Id}_{m_{1}}},\dots,e^{\widetilde{L}_{r} \log(z)}e^{\widetilde{\l}_{r}\times\mathrm{Id}_{m_{r}}}\right)$, the fundamental solution for the system ${ \d \widetilde{Y}(z)=\widetilde{C}(z)\widetilde{Y}(z)}$, which is defined in $\S \ref{sec12}$. Until the end of the subsection, we fix $\l\in \C^{*}$ with $d:=\arg(\l)  \in \R\setminus \widetilde{\Sigma} $. Theorem \ref{theo4} states that 
$$\widetilde{\D}=\left(\d -\widetilde{S}^{d}\left(\widetilde{f}_{m}\right)\right)\dots\left(\d -\widetilde{S}^{d}\left(\widetilde{f}_{1}\right)\right).$$
Until the end of the subsection, we fix $a,\e>0$ and we set
$$D_{d,a,\e}:=\left\{z\in \widetilde{\C}\left|z\in \overline{S}\left(d-\e,d+\e \right),\hbox{ with }|z|<a\right\}\right..$$ 
\pagebreak[3]
\begin{defi}\label{defi2}
Let $z\mapsto f(z,q)$ that is meromorphic on any compact subset of $D_{d,a,\e}$ for $q$ close to $1$. We say that $f\in \mathbb{B}_{d,a,\e}$ if for every $K$, compact subset of $D_{d,a,\e}$, there exist $q_{0}>1$, $N_{0}\in \N^{*}$, and $M_{0}>0$, such that for all $q\in ]1,q_{0}[$, $N>N_{0}$, $z\in K$, 
 $$\left|f\left(q^{-N}z,q\right)\right|<M_{0}.
$$
\end{defi}

\pagebreak[3]
\begin{rem}\label{rem4}
Note that $\mathbb{B}_{d,a,\e}$ is a ring.
\end{rem}

Until the end of the subsection, we fix $z\mapsto \L(z,q)\in \mathrm{GL}_{m}\left(\mathcal{M}(\C^{*})\right)$, family of diagonal matrices solution of (\ref{eq8}). We assume that:
\begin{trivlist}
\item \textbf{(H4)}  For all $1\leq j\leq m$, we have the uniform convergence in every  compact subset of~$D_{d,a,\e}$ to the function with no zeroes on $D_{d,a,\e}$ $${\lim\limits_{q \to 1}u_{\L,j,j}^{[\l]}(z,q)=:\widetilde{u}_{\L,j,j}^{d}(z)\in \mathcal{A}(d-\e,d+\e )}.$$

\item \textbf{(H5)} For all $1\leq j< m$, 
 $$\dfrac{u_{\Lambda,j+1,j+1}^{[\l]}}{z\sq\left(u_{\Lambda,j,j}^{[\l]}\right)}\in \mathbb{B}_{d,a,\e}.
$$
\end{trivlist}
\pagebreak[3]
\begin{rem}
In $\S \ref{sec42}$, we will show that on a convenient framework, 
\begin{itemize} 
\item we may chose the matrix $\L$ such that \textbf{(H4)} is satisfied,
\item there exists a direction $d\in \R$, and $a,\e>0$ such that \textbf{(H5)} is satisfied.
\end{itemize}
\end{rem}

Let ${(\hat{g}_{j,k}(z,q)):=\hat{G}(z,q)}$ be the formal matrix defined in Lemma \ref{lem1} that satisfies
$$C(z,q)=\hat{G}(z,q)\left[\mathrm{Diag}\left(z^{v_{0}(1+(q-1)f_{1})}t_{0}\left(1+(q-1)f_{1}\right),\dots,z^{v_{0}(1+(q-1)f_{m})}t_{0}\left(1+(q-1)f_{m}\right)\right)\right]_{\displaystyle\sq}.$$


We are going to see in the proof of Theorem \ref{theo2}, that we may define ${\left(\widetilde{u}^{d}_{\L,j,k}(z)\right)=\widetilde{U}^{d}_{\L}(z)\in \mathrm{GL}_{m}(\mathcal{A}(d-\e,d+\e ))}$, upper triangular matrix with diagonal terms defined in \textbf{(H4)}, as follows:
 for all $1\leq j<k\leq m$, for all $z\in D_{d,a,\e}$, set 
$$\widetilde{u}^{d}_{\L,j,k}(z):=\widetilde{u}^{d}_{\L,j,j}(z)\displaystyle\int_{0}^{z} \frac{\widetilde{u}^{d}_{\L,j+1,k}(t)}{\widetilde{u}^{d}_{\L,j,j}(t)}\frac{dt}{t}.$$
One may check that whether it is defined, $\widetilde{U}_{\L}^{d}$ is a fundamental solution for $$\d \widetilde{Y}(z)=\widetilde{C}(z)\widetilde{Y}(z).$$
\pagebreak[3]
\begin{rem}
Note that $\widetilde{U}_{\L}^{d}$ is not necessarily one of the fundamental solution defined in $\S \ref{sec12}$. However, since they are meromorphic fundamental solution for the same linear differential system, there exists $C_{\L,d}\in \mathrm{GL}_{m}(\C)$, such that $\widetilde{U}_{\L}^{d}C_{\L,d}$ equals to the corresponding fundamental solution defined in $\S \ref{sec12}$.\end{rem}
We may now state the main result of the subsection.

\pagebreak[3]\begin{theo}\label{theo2}
We have uniform convergence
 $$\lim\limits_{q \to 1}U_{\L}^{[\l]}(z,q)=\widetilde{U}_{\L}^{d}(z),$$ 
in every the compact subset of $D_{d,a,\e}$.
\end{theo}

Before proving the theorem, we prove two lemmas. We refer to $\S \ref{sec3}$ for the definition of the Jackson integral.

\pagebreak[3]
\begin{lem}\label{lem2}
Let $f\in \mathbb{B}_{d,a,\e}$ that converges uniformly to $\widetilde{f}\in \mathcal{A}(d-\e,d+\e)$ in every compact subset of $D_{d,a,\e}$, when $q$ goes to $1$. Then, for all $z\in D_{d,a,\e}$, for all $q$ close to~$1$, $\displaystyle\int_{0}^{z}f(t,q)d_{q}t$ is well defined and belongs to  $\mathbb{B}_{d,a,\e}$ (resp. $\displaystyle \int_{0}^{z}\widetilde{f}(t)dt$ is well defined and belongs to  $\mathcal{A}(d-\e,d+\e)$). Moreover, we have the uniform convergence 
$$\lim\limits_{q \to 1}\displaystyle \int_{0}^{z}f(t,q)d_{q}t=\displaystyle \int_{0}^{z}\widetilde{f}(t)dt,$$ 
in every compact subset of $D_{d,a,\e}$.
\end{lem}

\begin{proof}[Proof of Lemma \ref{lem2}]
Let $K$, be a compact subset of $D_{d,a,\e}$, $q_{0}>1$, $N_{0}\in \N^{*}$, and $M_{0}>0$, such that for all $q\in ]1,q_{0}[$, $N>N_{0}$, $z\in K$,  we have $$\left|f\left(q^{-N}z,q\right)\right|<M_{0}.$$ 
We now use the uniform convergence of $f$ on a compact subset of $D_{d,a,\e}$ that contains $\left\{ q^{-N}z\Big| z\in K, 1\leq N\leq N_{0},q\in ]1,q_{0}[\right\}$, to deduce that there exist 
$q_{1}>1$, and $M_{1}>0$, such that for all $q\in ]1,q_{1}[$, $N\in \N^{*}$, $z\in K$, we have $\left|f\left(q^{-N}z,q\right)\right|<M_{1}$.
It follows that for all $q$ close to $1$, for all $z\in K$, the Jackson integral 
$\int_{0}^{z}f(t,q)d_{q}t$ is well defined and for all $N\in \N^{*}$,
\begin{equation}\label{eq4}
\left|\int_{0}^{q^{-N}z}f(t,q)d_{q}t\right|\leq \left|q^{-N}z\right|M_{1}.\end{equation}
  Since $\int_{0}^{z}f(t,q)d_{q}t$ is a series of analytic functions, we obtain that for all $q$ close to $1$, $z\mapsto \int_{0}^{z}f(t,q)d_{q}t$ is analytic on $K$. To deduce $\int_{0}^{z}f(t,q)d_{q}t\in \mathbb{B}_{d,a,\e}$, we use (\ref{eq4}).\par 
 The facts that $\displaystyle \int_{0}^{z}\widetilde{f}(t)dt$ is well defined, belongs to  $\mathcal{A}(d-\e,d+\e)$ and the uniform convergence 
$$\lim\limits_{q \to 1}\displaystyle \int_{0}^{z}f(t,q)d_{q}t=\displaystyle \int_{0}^{z}\widetilde{f}(t)dt,$$ 
in every compact subset of $D_{d,a,\e}$, is a straightforward application of the dominated convergence theorem, we may apply because of (\ref{eq4}).
\end{proof}

\pagebreak[3]
\begin{lem}\label{lem7}
For all $1\leq j<k\leq m$, for all $z\in D_{d,a,\e}$, for all $q$ close to $1$, the following limit exists $$\lim\limits_{N \to +\infty}\dfrac{t_{0}\left(\hat{g}_{j,k}(z,q)\right) (q^{-N}z)^{v_{0}(\hat{g}_{j,k})}u_{\Lambda,k,k}^{[\l]}(q^{-N}z,q)}{u_{\Lambda,j,j}^{[\l]}(q^{-N}z,q)}=:c_{\L,j,k}^{[\l]}(z,q).$$
Moreover we have the uniform convergence $\lim\limits_{q \to 1}c_{\L,j,k}^{[\l]}(z,q)=0$ on the compact subset of $D_{d,a,\e}$.
\end{lem}

\begin{proof}[Proof of Lemma \ref{lem7}]
Let us fix $1\leq j<k\leq m$. Due to the hypothesis \textbf{(H5)}, we have the uniform convergence on the compact subset of $D_{d,a,\e}$ $$\lim\limits_{q \to 1}\lim\limits_{N \to +\infty}\dfrac{u_{\Lambda,k,k}^{[\l]}(q^{-N}z,q)}{u_{\Lambda,j,j}^{[\l]}(q^{-N+1}z,q)}=0.$$ We additionally  use $$\dfrac{u_{\Lambda,k,k}^{[\l]}(q^{-N}z,q)}{u_{\Lambda,j,j}^{[\l]}(q^{-N+1}z,q)}=\dfrac{u_{\Lambda,k,k}^{[\l]}(q^{-N}z,q)}{(1+(q-1)f_{j}(q^{-N}z,q))u_{\Lambda,j,j}^{[\l]}(q^{-N}z,q)},$$ and Remark \ref{rem5}, to deduce that for all $z\in D_{d,a,\e}$, $\lim\limits_{N \to +\infty}\dfrac{(q^{-N}z)^{v_{0}(\hat{g}_{j,k})}u_{\Lambda,k,k}^{[\l]}(q^{-N}z,q)}{u_{\Lambda,j,j}^{[\l]}(q^{-N}z,q)},$ exists and we have the uniform convergence $$\lim\limits_{q \to 1}\lim\limits_{N \to +\infty}\dfrac{(q^{-N}z)^{v_{0}(\hat{g}_{j,k})}u_{\Lambda,k,k}^{[\l]}(q^{-N}z,q)}{u_{\Lambda,j,j}^{[\l]}(q^{-N}z,q)}=0,$$ on the compact subset of $D_{d,a,\e}$. Using (\ref{eq12}) and the fact that \textbf{(H2)} is satisfied, we obtain that the following limit exists $\lim\limits_{q \to 1}t_{0}\left(\hat{g}_{j,k}(z,q)\right)$. This concludes the proof.
\end{proof}

\begin{proof}[Proof of Theorem \ref{theo2}]
Let $1\leq j\leq m$. The uniform convergence in every compact subset of $D_{d,a,\e}$: $$\lim\limits_{q \to 1}u_{\Lambda,j,j}^{[\l]}=\widetilde{u}^{d}_{\L,j,j}$$ is the assumption \textbf{(H4)}. If $m=1$ the proof is complete. Assume that $m\geq 2$. Let us fix ${1\leq j<k\leq m}$. Let us prove that $\widetilde{u}^{d}_{\L,j,k}(z)$ is well defined for $z\in D_{d,a,\e}$, ${\widetilde{u}^{d}_{\L,j,k}\in \mathcal{A}(d-\e,d+\e)}$ and that we have the uniform convergence in every compact subset of $D_{d,a,\e}$: $$\lim\limits_{q \to 1}u_{\Lambda,j,k}^{[\l]}=\widetilde{u}^{d}_{\L,j,k}.$$ 

Because of Lemma \ref{lem7}, we may apply  Proposition \ref{propo1} with the pairs $\{(\nu,k)| j\leq \nu <k\}$. Then, we obtain that $u_{\Lambda,j,k}^{[\l]}(z,q)$ equals 
 
$$ u_{\Lambda,j,k}^{[\l]}(z,q)=c_{\L,j,k}^{[\l]}(z,q)u_{\Lambda,j,j}^{[\l]}(z,q)+u_{\Lambda,j,j}^{[\l]}(z,q)F_{\L,j}^{[\l]}(z,q), $$
where 

$$ 
\begin{array}{ll}
F_{\L,\ell}^{[\l]}(z,q):=\displaystyle\int_{0}^{z} c_{\L,\ell+1,k}^{[\l]}(z,q)\frac{u_{\Lambda,\ell+1,\ell+1}^{[\l]}(t,q)}{u_{\Lambda,\ell,\ell}^{[\l]}(qt,q)}+\frac{u_{\Lambda,\ell+1,\ell+1}^{[\l]}(t,q)}{{u_{\Lambda,\ell,\ell}^{[\l]}(qt,q)}}F_{\L,\ell+1}^{[\l]}(t,q)\frac{d_{q}t}{t}&\hbox{ if } j\leq \ell \leq k-2\\
F_{\L,k-1}^{[\l]}(z,q):=\displaystyle\int_{0}^{z} \frac{u_{\Lambda,k,k}^{[\l]}(t,q)}{{u_{\Lambda,k-1,k-1}^{[\l]}(qt,q)}}\frac{d_{q}t}{t}.&
\end{array}
$$
Let us set, whenever the functions are well defined 
$$ 
\begin{array}{ll}
\widetilde{F}_{\L,\ell}^{d}(z):=\displaystyle\int_{0}^{z} \frac{\widetilde{u}^{d}_{\L,\ell+1,\ell+1}(t)}{\widetilde{u}^{d}_{\L,\ell,\ell}(t)}\widetilde{F}_{\L,\ell+1}^{d}(t)\frac{dt}{t}&\hbox{ if } j\leq \ell \leq k-2\\
\widetilde{F}_{\L,k-1}^{d}(z):=\displaystyle\int_{0}^{z} \frac{\widetilde{u}^{d}_{\L,k,k}(t)}{\widetilde{u}^{d}_{\L,k-1,k-1}(t)}\frac{dt}{t}.&
\end{array}
$$

Because of \textbf{(H4)} and \textbf{(H5)}, we may use Lemma \ref{lem2}, to deduce that\begin{itemize}
\item $F_{\L,k-1}^{[\l]}\in \mathbb{B}_{d,a,\e}$,
\item $\widetilde{F}_{\L,k-1}^{d}(z)$ is well defined for $z\in D_{d,a,\e}$ and  $\widetilde{F}_{\L,k-1}^{d}\in\mathcal{A}(d-\e,d+\e)$,
\item we have the uniform convergence in every compact subset of $D_{d,a,\e}$, $$\lim\limits_{q \to 1}F_{\L,k-1}^{[\l]}(z,q)=\widetilde{F}_{\L,k-1}^{d}(z).$$
\end{itemize} 

We now use \textbf{(H4)} and Lemma \ref{lem7}, to deduce that for all $j\leq \ell \leq k-2$, we have the  uniform convergence in every compact subset of~$D_{d,a,\e}$
$$\lim\limits_{q \to 1}c_{\L,\ell+1,k}^{[\l]}(z,q)\frac{u_{\Lambda,\ell+1,\ell+1}^{[\l]}(z,q)}{u_{\Lambda,\ell,\ell}^{[\l]}(qz,q)}=0$$
$$\lim\limits_{q \to 1}\frac{u_{\Lambda,\ell+1,\ell+1}^{[\l]}(z,q)}{zu_{\Lambda,\ell,\ell}^{[\l]}(qz,q)}=\frac{\widetilde{u}^{d}_{\L,\ell+1,\ell+1}(z)}{z\widetilde{u}^{d}_{\L,\ell,\ell}(z)}.$$

We use the above limits, whose functions involved in the left hand side are elements of $\mathbb{B}_{d,a,\e}$ (see \textbf{(H5)} and Lemma \ref{lem7}), and the fact that $\mathbb{B}_{d,a,\e}$ is a ring (Remark~\ref{rem4}), to apply  Lemma \ref{lem2} and deduce that if, $j<\ell < k$, and
\begin{itemize}
\item $F_{\L,\ell}^{[\l]}\in \mathbb{B}_{d,a,\e}$,
\item $\widetilde{F}_{\L,\ell}^{d}(z)$ is well defined for $z\in D_{d,a,\e}$ and  $\widetilde{F}_{\L,\ell}^{d}\in\mathcal{A}(d-\e,d+\e)$,
\item we have the uniform convergence in every compact subset of $D_{d,a,\e}$, $$\lim\limits_{q \to 1}F_{\L,\ell}^{[\l]}(z,q)=\widetilde{F}_{\L,\ell}^{d}(z),$$
\end{itemize}

  then 
\begin{itemize}
 \item  $F_{\L,\ell-1}^{[\l]}\in \mathbb{B}_{d,a,\e}$, 
 \item $\widetilde{F}_{\L,\ell-1}^{d}(z)$ is well defined for $z\in D_{d,a,\e}$ and  $\widetilde{F}_{\L,\ell-1}^{d}\in\mathcal{A}(d-\e,d+\e)$,
 \item we have the uniform convergence in every compact subset of $D_{d,a,\e}$ $$\lim\limits_{q \to 1}F_{\L,\ell-1}^{[\l]}(z,q)=\widetilde{F}_{\L,\ell-1}^{d}(z).$$ 

\end{itemize}
From what precede, $\widetilde{u}^{d}_{\L,j,k}(z)$ is well defined for $z\in D_{d,a,\e}$,  $\widetilde{u}^{d}_{\L,j,k}\in\mathcal{A}(d-\e,d+\e)$, and we have the uniform convergence in every compact subset of $D_{d,a,\e}$: $$\lim\limits_{q \to 1}u_{\Lambda,j,k}^{[\l]}=\widetilde{u}^{d}_{\L,j,k}.$$ 
\end{proof}

We finish this subsection by making some comparison between Theorem~\ref{theo2}, and two confluence results of the same nature.
\pagebreak[3]
\begin{rem}\label{rem1}
 We are now going to state \cite{DVZ}, Corollary 2.9, which is the particular case of \cite{DVZ}, Theorem~2.6, where the coefficients of the family of linear $q$-difference equations do not depend upon $q$. Let~$p=1/q$ and let~$\delta_{p}:=\frac{\sq^{-1}-\mathrm{Id}}{p-1}$, which converges formally to~$\d$ when~$p\to 1$. Let~${z\mapsto \hat{h}(z,p)\in \C\{z\}}$ that converges coefficientwise to~${\widetilde{h}(z)\in \C[[z]]}$ when~$p\to 1$. Assume the existence of~$b_{0},\dots,b_{m}\in \C[z]$, such that for all~$p$ close to~$1$, we have
$$
\left \{ \begin{array}{lll} 
b_{m}(z)\delta_{p}^{m}\hat{h}(z,p)+\dots+b_{0}(z)\hat{h}(z,p)&=&0\\\\
b_{m}(z)\d^{m}\widetilde{h}(z)+\dots+b_{0}(z)\widetilde{h}(z)&=&0.
\end{array}\right.$$
Moreover, assume that the series~$\hat{\mathcal{B}}_{1}\left(\widetilde{h}\right)$ belongs to~$\C\{z\}$ and is solution of a linear differential equation which is Fuchsian at~$0$ and infinity and has non resonant exponents at~$\infty$.\par 
 The authors of \cite{DVZ} conclude that for all convenient $d\in \R$,  $$\lim\limits_{p \to 1}\hat{h}(z,p)=\widetilde{S}^{d}\left(\widetilde{h}\right)(z),$$
uniformly on the compacts of~$\overline{S}\left(d-\frac{\pi}{2},d+\frac{\pi}{2}\right)$, where~$\widetilde{S}^{d}\left(\widetilde{h}\right)$ is the asymptotic solution of the linear differential equation that has been defined in $\S \ref{sec12}$.\end{rem}

\pagebreak[3]
\begin{rem}\label{rem2}
Let us now state \cite{D3}, Theorem 4.5, in the case where the coefficients of the family of linear $q$-difference equations do not depend upon $q$.
Let~${z\mapsto \hat{h}(z,q)\in \C[[z]]}$ that converges coefficientwise to~${\widetilde{h}(z)\in \C[[z]]}$ when~$q\to 1$. Assume the existence of polynomials~${b_{0},\dots,b_{m}\in \C[z]}$, such that for all~$q$ close to~$1$, we have
$$
\left \{ \begin{array}{lll} 
b_{m}(z)\delta_{q}^{m}\hat{h}(z,q)+\dots+b_{0}(z)\hat{h}(z,q)&=&0\\\\
b_{m}(z)\d^{m}\widetilde{h}(z)+\dots+b_{0}(z)\widetilde{h}(z)&=&0.
\end{array}\right.$$
In \cite{D3}, we prove that for all convenient~$d\in \R$, for all $q$ close to $1$, we may apply a $q$-deformation of the Borel-Laplace summation to $\hat{h}(z,q)$, to obtain ${z\mapsto S^{[d]}\left(\hat{h} \right)(z,q)\in \mathcal{M}(\C^{*})}$, solution of the same family of $\dq$-equations as $\hat{h}$. Moreover, we have $$\lim\limits_{q \to 1}S^{[d]}\left(\hat{h} \right)(z,q)=\widetilde{S}^{d}\left(\widetilde{h}\right)(z),$$
uniformly on the compacts of~$\overline{S}\left(d-\frac{\pi}{2\e},d+\frac{\pi}{2\e}\right)$, for some $\e>0$. Note that in general, $S^{[d]}\left(\hat{h} \right)$ has no link with the meromorphic solutions appearing in $\S \ref{sec22}$. 
\end{rem}

\pagebreak[3]
\subsection{An example.}\label{sec43}
In $\S \ref{sec42}$ we will explain how to apply Theorem \ref{theo2}. The goal of this section is to give an idea of the method we are going to use on an example. Since everything will be proved in $\S \ref{sec42}$ in full generalities, the proofs in this subsection will be omitted. \par 
Consider 
$$\widetilde{\D}:=\left(\d +2z^{-2}\right)\left(\d +z^{-1}\right)\d.$$
We have a formal solution given by the method of the constant variation:
$$\widetilde{U}(z):=\begin{pmatrix}
\exp(z^{-2})&\exp(z^{-2})\displaystyle\int_{0}^{z} \frac{\exp(t^{-1})}{\exp(t^{-2})}\frac{dt}{t}&\exp(z^{-2})\displaystyle\int_{t=0}^{z}\displaystyle\int_{u=0}^{t}\frac{\exp(t^{-1})}{\exp(t^{-2})} \frac{1}{\exp(u^{-1})}\frac{dt}{t}\frac{du}{u}\\
0&\exp(z^{-1})&\exp(z^{-1})\displaystyle\int_{0}^{z} \frac{1}{\exp(t^{-1})}\frac{dt}{t}\\
0&0&1
\end{pmatrix}.$$
We want to give an analytic meaning of $\widetilde{U}(z)$. 
The entries of the matrix $\widetilde{U}$ are analytic if the integrals converge. The integrals converge if and only if $\exp(\e z^{-1})$ and $\exp(\e z^{-2})$ tend to $0$ when $\e\to 0^{+}$. Therefore, the entries of $\widetilde{U}$ are analytic on the sector $\overline{S}\left(\frac{\pi}{2},\frac{3\pi}{4}\right)$ and on the sector $\overline{S}\left(\frac{5\pi}{4},\frac{3\pi}{2}\right)$. For the simplicity of the exposition, let us drop the second sector.\begin{figure}[ht]
\begin{center}
$\begin{tabular}{ccc}
\includegraphics[width=0.3\linewidth]{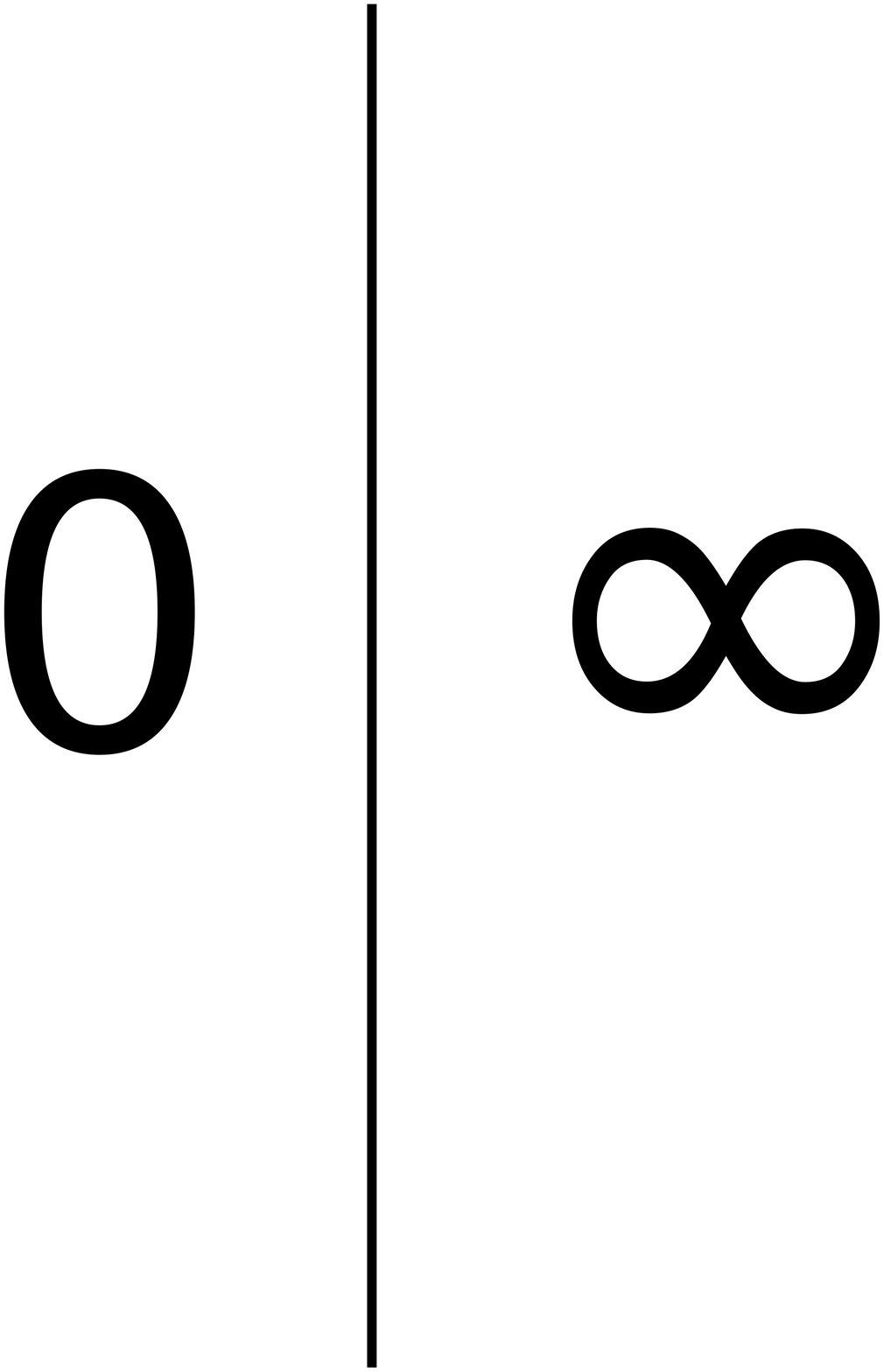}&\includegraphics[width=0.3\linewidth]{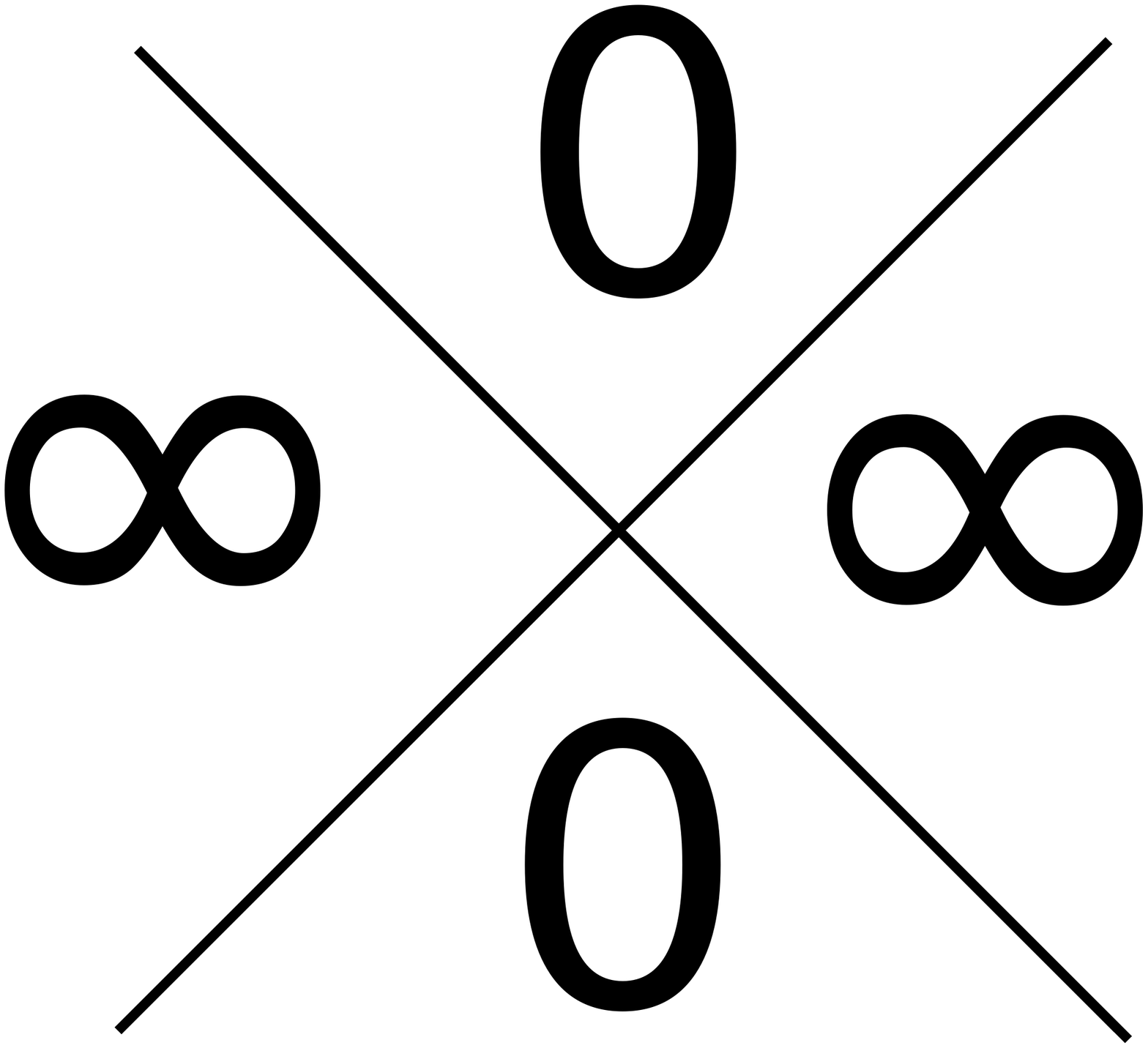}&\includegraphics[width=0.3\linewidth]{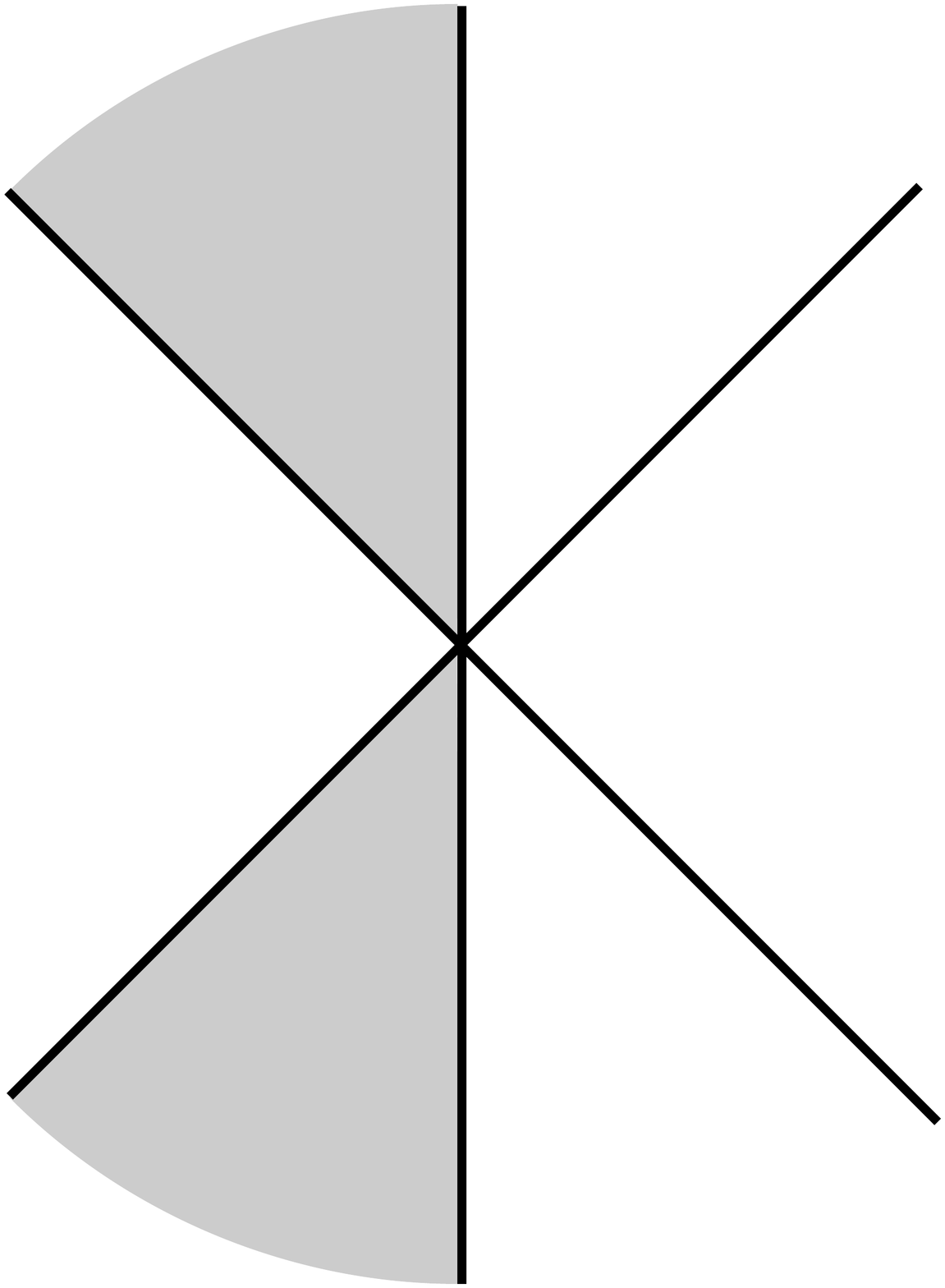}
\end{tabular}$
\caption{Limit of $\exp(\e z^{-1})$ when $\e\to 0^{+}$ (left). Limit of $\exp(\e z^{-2})$ when $\e\to 0^{+}$ (center). Domain of analyticity of the entries of $\widetilde{U}(z)$ (right).}
\end{center}
\end{figure} \par 
Let us now consider 
$$\D_{q}:=\left(\dq +\dfrac{(1+q)q^{-2}z^{-2}}{1+(q-1)(1+q)q^{-2}z^{-2}}\right)\left(\dq +\dfrac{q^{-1}z^{-1}}{1+(q-1)q^{-1}z^{-1}}\right)\dq.$$
Note that $\D_{q}$ is a $q$-deformation of  $\widetilde{\D}$.
For convenient $\l\in \C^{*}$, for $z\in \overline{S}\left(\frac{\pi}{2},\frac{3\pi}{4}\right)$, for some diagonal matrix $\Lambda$, we have the existence of $z\mapsto c_{\Lambda,1}^{[\l]}(z,q),c_{\Lambda,2}^{[\l]}(z,q),c_{\Lambda,3}^{[\l]}(z,q)$ meromorphic on $\overline{S}\left(\frac{\pi}{2},\frac{3\pi}{4}\right)$ and invariant under $\sq$, such that $U_{\Lambda}^{[\l]}(z,q)$ equals to
$$\begin{pmatrix}
e_{q^{2}}(z^{-2})&e_{q^{2}}(z^{-2})\displaystyle\int_{0}^{z} \frac{e_{q}(t^{-1})}{e_{q^{2}}(q^{-2}t^{-2})}\frac{d_{q}t}{t}&e_{q^{2}}(z^{-2})\displaystyle\int_{t=0}^{z}\displaystyle\int_{u=0}^{t}\frac{e_{q}(t^{-1})}{e_{q^{2}}(q^{-2}t^{-2})} \frac{1}{e_{q}(q^{-1}u^{-1})}\frac{d_{q}t}{t}\frac{d_{q}u}{u}\\
0&e_{q}(z^{-1})&e_{q}(z^{-1})\displaystyle\int_{0}^{z} \frac{1}{e_{q}(q^{-1}t^{-1})}\frac{d_{q}t}{t}\\
0&0&1
\end{pmatrix}C_{\Lambda}^{[\l]}(z,q),$$
where $C_{\Lambda}^{[\l]}(z,q):=\begin{pmatrix}
1&c_{\Lambda,1}^{[\l]}(z,q)&c_{\Lambda,2}^{[\l]}(z,q)\\
0&1&c_{\Lambda,3}^{[\l]}(z,q)\\
0&0&1 \end{pmatrix}$. One may check that assumptions of Theorem \ref{theo2} are satisfied. Therefore, there exists $a>0$ such that we have uniform convergence
 $$\lim\limits_{q \to 1}U_{\L}^{[\l]}(z,q)=\widetilde{U}(z),$$ 
in every the compact subset of $\left\{z\in \overline{S}\left(\frac{\pi}{2},\frac{3\pi}{4}\right)\Big| |z|<a\right\} $.

\pagebreak[3]
\subsection{$q$-deformation of meromorphic solutions of linear differential equations.}\label{sec42}

Let us consider a linear differential equation in coefficients in $\C(\{z\})$
$$\widetilde{\D}:=\left(\d -\widetilde{f}_{m}(z)\right)\dots\left(\d -\widetilde{f}_{1}(z)\right),$$
with $m\geq 2$, for all $1\leq j\leq m$, ${\widetilde{f}_{j}\in \C((z))}$, ${v_{0}\left( \widetilde{f}_{1}\right)<\dots< v_{0}\left( \widetilde{f}_{m}\right)}$ and $v_{0}\left( \widetilde{f}_{m-1}\right)< 0$. As in $\S \ref{sec1}$, we are going to see the equation as a system $\d \widetilde{Y}(z)=\widetilde{C}(z)\widetilde{Y}(z)$.

\pagebreak[3]
\begin{rem}
Note that the Newton polygon of $\widetilde{\D}$ have $m$ distinct slopes with multiplicities one. Moreover, every linear differential equation
$$ \displaystyle \sum_{\ell=0}^{m-1}\widetilde{a}_{\ell}\d^{\ell},$$
with $m\geq 2$, $\widetilde{a}_{0},\dots,\widetilde{a}_{m-1}\in \C(\{z\})$, and ${v_{0}\left( \widetilde{a}_{m-1}\right)>\dots>v_{0}\left( \widetilde{a}_{1}\right)}$, $v_{0}\left( \widetilde{a}_{1}\right)\leq v_{0}\left( \widetilde{a}_{0}\right)$ is of the wished form.  
\end{rem}

For $1\leq j\leq m$, and $d\in \R$, that is not one of the singular directions of ${\d \widetilde{Y}(z)=\widetilde{C}(z)\widetilde{Y}(z)}$, see Theorem \ref{theo4}, let us chose $\widetilde{u}^{d}_{j,j}$, non zero meromorphic solution of ${\d\widetilde{u}^{d}_{j,j}=\widetilde{S}^{d}\left(\widetilde{f}_{j}\right)\widetilde{u}^{d}_{j,j}}$. Let us introduce, whether it is defined, ${\left(\widetilde{u}^{d}_{j,k}(z)\right):=\widetilde{U}^{d}(z)}$, upper triangular matrix with, 
 for all $1\leq j<k\leq m$,
$$\widetilde{u}^{d}_{j,k}(z):=\widetilde{u}^{d}_{j,j}(z)\displaystyle\int_{0}^{z} \frac{\widetilde{u}^{d}_{j+1,k}(t)}{\widetilde{u}^{d}_{j,j}(t)}\frac{dt}{t}.$$
We may check that if it is defined, $\widetilde{U}^{d}(z)$ is a meromorphic fundamental solution for $\d \widetilde{Y}(z)=\widetilde{C}(z)\widetilde{Y}(z)$. 
 The goal of this subsection is to 
 \begin{itemize}
\item  prove that $\widetilde{U}^{d}(z)$ is well defined for some $d\in \R$, and some $z\in \C^{*}$,
\item  construct a family of linear $q$-difference equations $\D_{q}$ that discretizes $\widetilde{\D}$,
\item build a family of meromorphic fundamental solution for the family of systems attached to $\D_{q}$, that is, for $q$ close to $1$ fixed, one of the meromorphic fundamental solution defined in $\S \ref{sec22}$, and that converges to $\widetilde{U}^{d}(z)$ when $q$ goes to $1$.
 \end{itemize}

\pagebreak[3]
\begin{center}
\textbf{Step 1: Domain of definition of $\widetilde{U}^{d}(z)$.}
\end{center} 

\begin{lem}\label{lem6}
Let us write  ${\displaystyle \sum_{\ell=\mu_{j}}^{\infty}\widetilde{f}_{j,\ell}z^{\ell}:=\widetilde{f}_{j}}$, with $\widetilde{f}_{j,\mu_{j}}\neq 0$. There exist $\e>0$, $d\in \R$, such that for all ${z\in \overline{S}(d-\e,d+\e)}$, and for all~${1\leq j<m}$,
$$\lim\limits_{\substack{ x \to 0^{+}\\ x\in \R_{>0}}}\left|\frac{\exp\left(\widetilde{f}_{j+1,0}\log(xz)\right)\displaystyle\prod_{\ell=\mu_{j+1}}^{-1}\exp\left(-\ell^{-1}\widetilde{f}_{j+1,\ell}(xz)^{\ell}\right)}{\exp\left(\left(\widetilde{f}_{j,0}+1\right)\log(xz)\right)\displaystyle\prod_{\ell=\mu_{j}}^{-1}\exp\left(-\ell^{-1}\widetilde{f}_{j,\ell}(xz)^{\ell}\right)}\right|=0.$$
\end{lem}

\begin{proof}
Since ${v_{0}\left( \widetilde{f}_{1}\right)<\dots< v_{0}\left( \widetilde{f}_{m}\right)\leq 0}$, it is equivalent to prove the existence of $\e>0$, $d\in \R$, such that for all ${z\in \overline{S}(d-\e,d+\e)}$, such that for all ${1\leq j< m}$,
$$\lim\limits_{\substack{ x \to 0^{+}\\ x\in \R_{>0}}}\left|\exp\left(-\mu_{j}^{-1}\widetilde{f}_{j,\mu_{j}}(xz)^{\mu_{j}}\right)\right|=\infty.$$
For all ${1\leq j< m}$, $\lim\limits_{\substack{ x \to 0^{+}\\ x\in \R_{>0}}}\left|\exp\left(-\mu_{j}^{-1}\widetilde{f}_{j,\mu_{j}}(xz)^{\mu_{j}}\right)\right|=\infty$ if and only if $\arg(z)$ belongs to the open set $\displaystyle\bigcup_{k=-\infty}^{\infty}\left]\arg\left(-\widetilde{f}_{j,\mu_{j}}\right)+\frac{2k\pi}{\mu_{j}}-\frac{\pi}{2\mu_{j}},\arg\left(-\widetilde{f}_{j,\mu_{j}}\right)+\frac{2k\pi}{\mu_{j}}+\frac{\pi}{2\mu_{j}}\right[$.
Let ${K:=\mu_{1}\times\dots\times\mu_{m-1}}$. Let $\e_{0}>0$ such that for all $\ell\in \Z$, $\frac{\ell\pi}{ K}+\e_{0}$ does not belong to the border of $$\displaystyle \bigcap_{j=1}^{m-1}\left(\bigcup_{k=0}^{\mu_{j}-1}\left]\arg\left(-\widetilde{f}_{j,\mu_{j}}\right)+\frac{2k\pi}{\mu_{j}}-\frac{\pi}{2\mu_{j}},\arg\left(-\widetilde{f}_{j,\mu_{j}}\right)+\frac{2k\pi}{\mu_{j}}+\frac{\pi}{2\mu_{j}}\right[\right) . $$

Then, to prove the lemma, it is sufficient to prove the existence of $\ell\in \Z$, such that 
$$\frac{\ell\pi}{ K}+\e_{0}\in\displaystyle \bigcap_{j=1}^{m-1}\left(\bigcup_{k=0}^{\mu_{j}-1}\left]\arg\left(-\widetilde{f}_{j,\mu_{j}}\right)+\frac{2k\pi}{\mu_{j}}-\frac{\pi}{2\mu_{j}},\arg\left(-\widetilde{f}_{j,\mu_{j}}\right)+\frac{2k\pi}{\mu_{j}}+\frac{\pi}{2\mu_{j}}\right[\right) . $$
If $m=2$, the proof is completed. Assume that $m>2$. For ${1\leq j< m}$, let $\Sigma_{j}$, be the set of integers $\ell$ in $\Z$, such that 
$$\frac{\ell\pi}{ K}+\e_{0}\in \bigcup_{k=0}^{\mu_{j}-1}\left]\arg\left(-\widetilde{f}_{j,\mu_{j}}\right)+\frac{2k\pi}{\mu_{j}}-\frac{\pi}{2\mu_{j}},\arg\left(-\widetilde{f}_{j,\mu_{j}}\right)+\frac{2k\pi}{\mu_{j}}+\frac{\pi}{2\mu_{j}}\right[.$$
 Let us write $2K=2\times2^{a_{0}}\times p_{1}^{a_{1}}\times\dots \times p_{r}^{a_{r}}$, the decomposition of $2K$ in product of prime numbers. For $p$ prime number, let $v_{p}$ be $p$-adic valuation.\par 
  We claim that for ${1\leq j< m}$, there exists $\ell_{j}\in \Sigma_{j}$, such that for all $\nu \in \{0,\dots,r\}$, for all ${1\leq j<k< m}$, the projections of $\ell_{j} \mod p_{\nu}^{v_{p_{\nu}}(\frac{K}{\mu_{j}})}$ and $\ell_{k} \mod p_{\nu}^{v_{p_{\nu}}(\frac{K}{\mu_{k}})}$ on $\Z/p_{\nu}^{a_{\nu}}\Z$ are equal. For ${1\leq j< m}$ let us fix an arbitrary $\ell_{j}\in \Sigma_{j}$.  Let $\nu:=0$ and $p_{0}:=2$. We do the following operation. Let $1\leq j_{\nu} <m$, such that $v_{p_{\nu}}\left(\frac{K}{\mu_{j_{\nu}}}\right)=\displaystyle \max_{1\leq j <m}\left(v_{p_{\nu}}\left(\frac{K}{\mu_{j}}\right) \right)$. By construction, $\Sigma_{j}$ is an union successive sets of consecutive $\frac{K}{\mu_{j}}$ integers separated by $\frac{K}{\mu_{j}}$ consecutive integers. Then, for all $1\leq j <m$ with $j\neq j_{\nu}$, there exists $\kappa_{j,\nu}\in \Z$, such that $\ell_{j}':=\ell_{j}+\kappa_{j,\nu} \displaystyle \prod_{i=0}^{\nu-1}p_{i}^{v_{p_{i}}(\frac{K}{\mu_{j}})}\in  \Sigma_{j}$, and such that for all ${1\leq j<k< m}$, for all $0\leq i \leq \nu$, the projections of $\ell_{j}' \mod p_{i}^{v_{p_{i}}(\frac{K}{\mu_{j}})}$ and $\ell_{k}' \mod p_{i}^{v_{p_{i}}(\frac{K}{\mu_{k}})}$ on $\Z/p_{i}^{a_{i}}\Z$ are equal. It follows that we may replace $\ell_{j}$ by $\ell_{j}'$, and reduce to the case where for all ${1\leq j<k< m}$, for all $0\leq i \leq \nu$, the projections of $\ell_{j} \mod p_{i}^{v_{p_{i}}(\frac{K}{\mu_{j}})}$ and $\ell_{k} \mod p_{i}^{v_{p_{i}}(\frac{K}{\mu_{k}})}$ on $\Z/p_{i}^{a_{i}}\Z$ are equal. We do the same operation for $\nu=1,\dots,\nu=r$. This proves our claim.\par 
Let $1\leq j_{-1} <m$, such that $v_{2}\left(\frac{K}{\mu_{j_{-1}}}\right)=\displaystyle \max_{1\leq j <m}\left(v_{2}\left(\frac{K}{\mu_{j}}\right) \right)$. We remind that $\Sigma_{j}$ is an union successive sets of consecutive $\frac{K}{\mu_{j}}$ integers separated by $\frac{K}{\mu_{j}}$ consecutive integers. For all $1\leq j <m$ with $j\neq j_{-1}$, we again replace $\ell_{j}$ by $\ell_{j}+\kappa_{j} \frac{2K}{\mu_{j}}\in  \Sigma_{j}$ for some $\kappa_{j}\in \Z$, to reduce to the case where for all ${1\leq j<k< m}$, for $\nu \in \{1,\dots,r\}$ (resp. for $\nu=0$), the projections of $\ell_{j} \mod p_{\nu}^{v_{p_{\nu}}(\frac{K}{\mu_{j}})}$ and $\ell_{k} \mod p_{\nu}^{v_{p_{\nu}}(\frac{K}{\mu_{k}})}$ on $\Z/p_{\nu}^{a_{\nu}}\Z$ (resp. on $\Z/2^{a_{0}+1}\Z$), are equal.\par 
Due to the Chinese reminder theorem, there exists an integer $0\leq n_{0}<2K$, such that for all ${1\leq j< m}$, $$n_{0}\equiv \ell_{j} \mod \frac{2K}{\mu_{j}}\Z.$$ Then, we have 
$$\frac{n_{0}\pi}{K}+\e_{0}\in\displaystyle \bigcap_{j=1}^{m-1}\left(\bigcup_{k=0}^{\mu_{j}-1}\left]\arg\left(-\widetilde{f}_{j,\mu_{j}}\right)+\frac{2k\pi}{\mu_{j}}-\frac{\pi}{2\mu_{j}},\arg\left(-\widetilde{f}_{j,\mu_{j}}\right)+\frac{2k\pi}{\mu_{j}}+\frac{\pi}{2\mu_{j}}\right[\right) . $$
This concludes the proof.
\end{proof}

Without loss of generalities, we may assume that the real number $d\in \R$ of Lemma~\ref{lem6} is not one of the singular directions of $\d \widetilde{Y}(z)=\widetilde{C}(z)\widetilde{Y}(z)$ defined in Theorem \ref{theo4}. It follows that: 
\pagebreak[3]
\begin{coro}
There exist $d\in \R$, and $\e>0$, such that $$\widetilde{U}^{d}(z)\in \mathrm{GL}_{m}(\mathcal{A}(d-\e,d+\e )).$$
\end{coro}

\pagebreak[3]
\begin{center}
\textbf{Step 2: Construction of the family of $q$-difference equations.}
\end{center}
We would like to define a convenient family of $q$-difference equations that discretizes $\widetilde{\D}$, ${\D_{q}:=\left(\dq -f_{m}(z,q)\right)\dots\left(\dq -f_{1}(z,q)\right)}$ with for all $1\leq j\leq m$, $z\mapsto f_{j}(z,q)\in \C(\{z\})$.
But before, let us introduce some notations. 
Let $p=1/q$ and define the~$p$-exponential:
$$e_{p}(z):=\displaystyle \sum_{n=0}^{\infty}\dfrac{z^{n}}{[n]_{p}^{!}}\in \C\{z\},$$
where ${[n]_{p}^{!}:=\prod_{l=0}^{n} [l]_{p}}$, ${[l]_{p}:=\left(1+...+p^{l-1}\right)}$. Its radius of convergence is $\frac{1}{1-p}$. For all $z\in \C$ with $|z|<\frac{1}{1-p}$ one may check that we have
$|e_{p}(z)|\leq e_{p}(|z|)$. Hence, the dominated convergence theorem gives that we have the uniform convergence in the compact subset of $\C$
\begin{equation}\label{eq5}
\displaystyle\lim\limits_{p\to 1}e_{p}(z)=\exp(z).\end{equation}
Let us also define the $p$-Gamma function 
$$\Gamma_p(z) = (1-p)^{1-z}\prod_{n=0}^\infty 
\frac{1-p^{n+1}}{1-p^{n+z}}.$$ It converges to the classical Gamma function when $p\to 1$  and satisfies 
$$\Gamma_p(z+1)=\frac{1-p^{z}}{1-p}\Gamma_p(z).$$
For $\ell \in \N^{*}$, set 
$$
\begin{array}{llll}
\hat{\mathcal{B}}_{p,\ell}:&\C\{z\}&\longrightarrow&\C\{\z\}\\
&\displaystyle\sum_{n\geq 0} a_{n}z^{n}&\longmapsto& \displaystyle\sum_{n\geq 0} \frac{a_{n}}{\Gamma_p\left(1+\frac{n}{\ell}\right)}\z^{n}.
\end{array}
$$
Let us now define $f_{1},\dots,f_{m}$. Let $\widetilde{f}_{j}^{>0}:=\displaystyle \sum_{\ell=1}^{\infty}\widetilde{f}_{j,\ell}z^{\ell}$, and  $\widetilde{f}_{j}^{\leq 0}:=\displaystyle \sum_{\ell=\mu_{j}}^{0}\widetilde{f}_{j,\ell}z^{\ell}$ so that ${\widetilde{f}_{j}=\widetilde{f}_{j}^{>0}+\widetilde{f}_{j}^{\leq 0}}$. Since $\widetilde{\D}$ is a linear differential equation in coefficients in $\C(\{z\})$, we may apply Theorem \ref{theo4}, to deduce for all $1\leq j\leq m$, the existence of a decomposition 
$$\widetilde{f}_{j}^{>0}=\widetilde{f}_{j,1}^{>0}+\dots+\widetilde{f}_{j,r}^{>0},$$
with $\widetilde{f}_{j,k}^{>0}\in \widetilde{\mathbb{S}}_{\ell_{j,k}}^{d}$, for some $\ell_{j,k}\in \N^{*}$. 

For all $1\leq j\leq m$, let us define $f_{j}^{>0}=f_{j,1}^{>0}+\dots+f_{j,r}^{>0}$, where for all $1\leq k\leq r$, the function ${z\mapsto f_{j,k}^{>0}(z,q)\in \C\{z\}}$ is defined such that for all $q$ close to $1$, $${g_{j,k}:=\hat{\mathcal{B}}_{p,\ell_{j,k}}\left(f_{j,k}^{>0} \right)=\hat{\mathcal{B}}_{\ell_{j,k}}\left(\widetilde{f}_{j,k}^{>0} \right)}.$$ 

Let $1+(q-1)f_{m}^{\leq 0}:=1+(q-1)\widetilde{f}_{m}^{\leq 0}$ and for all $1\leq j< m$, let us set ${z\mapsto f_{j}^{\leq 0}(z,q)\in \C\{z^{-1}\}}$, such that for all $1\leq j<m$,

\begin{equation}\label{eq17}
\frac{1+(q-1)f_{j+1}^{\leq 0}}{q\left(1+(q-1)f_{j}^{\leq 0}\right)}=1+(q-1)\left(\widetilde{f}_{j+1}^{\leq 0}-\widetilde{f}_{j}^{\leq 0} -1\right).
\end{equation}
 
Finally, we set $z\mapsto f_{j}(z,q)\in \C(\{z\})$, such that 
$$1+(q-1)f_{j}:=\left(1+(q-1)f_{j}^{>0}\right)\left(1+(q-1)f_{j}^{\leq 0}\right).$$

\pagebreak[3]
\begin{lem}\label{lem4}
There exists $a>0$ such that the functions $f_{1}^{>0},\dots,f_{m}^{>0}$ belong to $ \mathbb{B}_{d,a,\e}$. Moreover, for all ${1\leq j\leq m}$, we have the uniform convergence $\displaystyle \lim\limits_{q\to 1}f_{j}=\widetilde{S}^{d}\left(\widetilde{f}_{j}\right)$ in every compact subset of $D_{d,a,\e}$.
\end{lem}

\begin{proof}
By construction, it is sufficient to prove the existence of $a>0$ such that the functions $f_{1}^{>0},\dots,f_{m}^{>0}$ belong to $ \mathbb{B}_{d,a,\e}$ and for all ${1\leq j\leq m}$, we have the uniform convergence $\displaystyle \lim\limits_{q\to 1}f^{>0}_{j}=\widetilde{S}^{d}\left(\widetilde{f}^{>0}_{j}\right)$ in every compact subset of $D_{d,a,\e}$.\\ \par 

In (2.11.1) of \cite{DVZ}, we see that for all $1\leq j\leq m$ and $1\leq k \leq r$ we have
$$f_{j,k}^{>0}(z,q)= \displaystyle \int_{0}^{\frac{z^{\ell_{j,k}}}{q-1}}g_{j,k}\left(\z^{1/\ell_{j,k}}\right)e_{p}\left(p\z/z^{\ell_{j,k}}\right)d_{q}\z.$$
Moreover, we have
$$\widetilde{S}^{d}\left(\widetilde{f}_{j,k}^{>0}(z,q)\right)=\displaystyle \int_{0}^{\infty e^{id\ell_{j,k}}}g_{j,k}\left(\z^{1/\ell_{j,k}}\right)\exp\left(\z/z^{\ell_{j,k}}\right)d\z.$$
In \cite{DVZ}, Page 11, we see that for all $f\in  \C\{z\}$, for all $p<1$, we have the equality 
\begin{equation}\label{eq9}
\displaystyle \int_{0}^{\frac{z^{\ell_{j,k}}}{q-1}}\hat{\mathcal{B}}_{p,\ell_{j,k}}\left(f\right)\left(\z^{1/\ell_{j,k}}\right)e_{p}\left(p\z/z^{\ell_{j,k}}\right)d_{q}\z=f(z).
\end{equation}
Using (\ref{eq5}), the fact that for all $z\in \C$ with $|z|<\frac{1}{1-p}$, we have ${|e_{p}(z)|\leq e_{p}(|z|)}$, and (\ref{eq9}), we may apply the dominated convergence theorem, in order to obtain that for all $a>0$ sufficiently small, we have the uniform convergence ${\displaystyle \lim\limits_{q\to 1}f^{>0}_{j}=\widetilde{S}^{d}\left(\widetilde{f}^{>0}_{j}\right)}$ in every compact subset of $D_{d,a,\e}$. Since for every $a>0$, $\mathbb{B}_{d,a,\e}$ is a ring, see Remark \ref{rem4}, it is now sufficient to prove the existence of $a>0$ such that for all $1\leq j\leq m$ and $1\leq k \leq r$, we have  $f_{j,k}^{>0}\in \mathbb{B}_{d,a,\e}$.\par 
Let us fix $1\leq j\leq m$ and $1\leq k \leq r$. Because of the definition of $\widetilde{\mathbb{S}}_{\ell_{j,k}}^{d}$ there exist $J_{j,k},L_{j,k}>0$ such that 
$$|g_{j,k}(\z)|\leq J_{j,k}\exp\left(L_{j,k}|\z|^{\ell_{j,k}}\right).$$
We remind that for all $z\in \C$ with $|z|<\frac{1}{1-p}$  we have
$ |\exp(z)|\leq e_{p}(|z|)$. Therefore, for all $q$ close to $1$, for all $z\in D_{d,a,\e}$, we obtain
$$\left| f_{j,k}^{>0}(z,q)\right|\leq \displaystyle \int_{0}^{\frac{z^{\ell_{j,k}}}{q-1}}J_{j,k}e_{p}\left(L_{j,k}^{1/\ell_{j,k}}|\z|\right)e_{p}\left(|p\z/z^{\ell_{j,k}}|\right)\frac{|\z|}{\z}d_{q}\z.$$
Let $a>0$ with for all $1\leq j\leq m$ and $1\leq k \leq r$, $a^{\ell_{j,k}}<1/L_{j,k}$. Let us remark that for all $p<1$ for all $z\in D_{d,a,\e}$, we have the equality of functions
\begin{equation}\label{eq6}
\hat{\mathcal{B}}_{p,\ell_{j,k}}\left(\frac{J_{j,k}}{1- L_{j,k}z^{\ell_{j,k}} }\right)=J_{j,k}e_{p}\left(L_{j,k}\z^{\ell_{j,k}}\right).
\end{equation}
  Therefore, if we combine (\ref{eq9}) and (\ref{eq6}), we find that for all $z\in D_{d,a,\e}$, for all $q$ close to~$1$, there exist $\t(z,q),\t'(z,q)$ real functions, such that (note that the right hand side is well defined due to $a^{\ell_{j,k}}<1/L_{j,k}$)  
\begin{equation}\label{eq11}
\begin{array}{ll}
\left|f_{j,k}^{>0}(z,q)\right|&\leq \displaystyle \int_{0}^{\frac{z^{\ell_{j,k}}}{q-1}}J_{j,k}e_{p}\left(L_{j,k}^{1/\ell_{j,k}}|\z|\right)e_{p}\left(|p\z/z^{\ell_{j,k}}|\right)\frac{|\z|}{\z}d_{q}\z\\
&\leq \left|\frac{J_{j,k}}{1- L_{j,k}e^{i\t(z,q)}(e^{i\t'(z,q)}z)^{\ell_{j,k}} }\right|.
\end{array}
\end{equation}
Since $a^{\ell_{j,k}}<1/L_{j,k}$, this proves that for all $1\leq j\leq m$ and $1\leq k \leq r$, we have  $f_{j,k}^{>0}\in \mathbb{B}_{d,a,\e}$. This concludes the proof.
\end{proof}

\pagebreak[3]
\begin{center}
\textbf{Step 3: Construction of the family of solutions}
\end{center}

We consider $$
\left\{\begin{array}{lll}
\D_{q}&=&\left(\dq -f_{m}(z,q)\right)\dots\left(\dq -f_{1}(z,q)\right)\\
\widetilde{\D}&=&\left(\d -\widetilde{f}_{m}(z)\right)\dots\left(\d -\widetilde{f}_{1}(z)\right),
\end{array} \right. 
$$
where $\widetilde{f}_{1},\dots,\widetilde{f}_{m}\in \C((z))$ have been defined in the beginning of the subsection and $f_{1},\dots,f_{m}$ have been defined before Lemma \ref{lem4}. As in $\S \ref{sec1},\S \ref{sec2}$, we are going to see the equations as systems:
$$\left\{\begin{array}{lll}
\dq Y(z,q)&=&C(z,q)Y(z,q)\\ 
 \d \widetilde{Y}(z)&=&\widetilde{C}(z)\widetilde{Y}(z).
 \end{array} \right. $$
By construction, we obtain that the assumptions (\textbf{H1}) to (\textbf{H3}) of $\S \ref{sec41}$ are satisfied. Without loss of generalities, we may assume that for all $\l\in \C^{*}$ with $\arg(\l)=d$, for all 
$q$ close to $1$,  for all ${z\mapsto \L(z,q)\in \mathrm{GL}_{m}\left(\mathcal{M}(\C^{*})\right)}$, family of diagonal matrices solution of (\ref{eq8}), we may consider $U_{\L}^{[\l]}(z,q)$, the fundamental solution for the system ${\dq Y(z,q)=C(z,q)Y(z,q)}$, which is defined in $\S \ref{sec22}$.
 Until the end of the subsection, we fix $\l\in \C^{*}$ with $d=\arg(\l)$. We remind that for all $1\leq j\leq m$, $\widetilde{u}^{d}_{j,j}$, is a non zero meromorphic solution of $\d\widetilde{u}^{d}_{j,j}=\widetilde{S}^{d}\left(\widetilde{f}_{j}\right)\widetilde{u}^{d}_{j,j}$.

\pagebreak[3]
\begin{lem}\label{lem3}
There exists  $z\mapsto \L(z,q)\in \mathrm{GL}_{m}\left(\mathcal{M}(\C^{*})\right)$, family of diagonal matrices solution of (\ref{eq8}), 
such that for all $1\leq j\leq m$, we have the uniform convergence ${\lim\limits_{q \to 1}u_{\L,j,j}^{[\l]}(z,q)=\widetilde{u}_{j,j}^{d}(z)\in \mathcal{A}(d-\e,d+\e )}$,
in every compact subset of $D_{d,a,\e}$.
\end{lem}

\begin{proof}
Since two non zero solutions of $\d\widetilde{y}_{j}=\widetilde{S}^{d}\left(\widetilde{f}_{j}\right)\widetilde{y}_{j}$ equals up to a multiplication by a non zero complex number, it is sufficient to prove that for all $1\leq j\leq m$, there exists ${z\mapsto y_{j}(z,q)\in \mathcal{M}(\C^{*},0)}$, that satisfies $\sq y_{j}(z,q)=(1+(q-1)f_{j}(z,q))y_{j}(z,q)$, such that we have the uniform convergence ${\lim\limits_{q \to 1}y_{j}(z,q)=:\widetilde{y}_{j}(z)\in \mathcal{A}(d-\e,d+\e )}$,
in every the compact subset of $D_{d,a,\e}$.\par

Let us fix $1\leq j\leq m$. Let us define $z\mapsto w_{j}(z,q)\in \mathcal{M}(\C^{*},0)\cap \C\left\{z^{-1}\right\}$ with constant term equals to $1$ that satisfies 
$$\frac{\sq w_{j}}{w_{j}}=\frac{\left(1+(q-1)f_{j}^{\leq 0}\right)\left(\displaystyle\L_{q,1+(q-1)\widetilde{f}_{j,0}}\displaystyle\prod_{\ell=\mu_{j}}^{-1}e_{p^{\ell}}\left(-\ell^{-1}\widetilde{f}_{j,\ell}z^{\ell}\right)\right)}{\sq\left(\displaystyle\L_{q,1+(q-1)\widetilde{f}_{j,0}}\displaystyle\prod_{\ell=\mu_{j}}^{-1}e_{p^{\ell}}\left(-\ell^{-1}\widetilde{f}_{j,\ell}z^{\ell}\right)\right)}.$$
By construction, we have
$$\frac{\sq\left(w_{j}(z,q)\displaystyle\L_{q,1+(q-1)\widetilde{f}_{j,0}}\displaystyle\prod_{\ell=\mu_{j}}^{-1}e_{p^{\ell}}\left(-\ell^{-1}\widetilde{f}_{j,\ell}z^{\ell}\right)\right)}{w_{j}(z,q)\displaystyle\L_{q,1+(q-1)\widetilde{f}_{j,0}}\displaystyle\prod_{\ell=\mu_{j}}^{-1}e_{p^{\ell}}\left(-\ell^{-1}\widetilde{f}_{j,\ell}z^{\ell}\right)}=1+(q-1)f_{j}^{\leq 0}.$$
Let us fix $K$, a compact subset of $D_{d,a,\e}$, and for $N\in \N^{*}$,  $z\in K$, set
$$ y_{j,N}(z,q):=w_{j}(z,q)\displaystyle\L_{q,1+(q-1)\widetilde{f}_{j,0}}\displaystyle\prod_{\ell=\mu_{j}}^{-1}e_{p^{\ell}}\left(-\ell^{-1}\widetilde{f}_{j,\ell}z^{\ell}\right)\displaystyle\prod_{\nu=-N}^{-1}1+(q-1)f_{j}^{>0}(q^{\nu}z,q).$$
Lemma \ref{lem4} tell us that $f_{j}^{>0}\in \mathbb{B}_{d,a,\e}$. Since its valuation is at least $1$, we find that there exists $\a>0$ such that for all $\nu<0$, for all $z\in K$, for all $q$ close to $1$, ${|1+(q-1)f_{j}^{>0}(q^{\nu}z,q)|<1+(q-1)\a|q^{\nu}z|}$. Let us introduce the $q$-exponential $$e_{q}(z):=\displaystyle \sum_{n=0}^{\infty}\dfrac{z^{n}}{[n]_{q}^{!}}=\displaystyle \prod_{n=0}^{\infty}\left(1+(q-1)q^{-n-1}z\right),$$
where ${[n]_{q}^{!}:=\prod_{l=0}^{n} [l]_{q}}$, ${[l]_{q}:=\left(1+...+q^{l-1}\right)}$.
 It is analytic on~$\C$, with simple zeros on the discrete~$q$-spiral~$\frac{q^{\N^{*}}}{1-q}$ and satisfies~${\dq e_{q}(z)=ze_{q}(z)}$. Using the infinite product expression of $e_{q}$, one finds that for all $z\in K$, for all ${N\in \N^{*}}$, for all $q$ close to $1$, 
\begin{equation}\label{eq19}
 \left|\displaystyle\prod_{\nu=-N}^{-1}1+(q-1)f_{j}^{>0}(q^{\nu}z,q)\right|\leq \displaystyle\prod_{\nu=-\infty}^{-1}1+(q-1)\a|q^{\nu}z|= e_{q}(\a|z|).
 \end{equation}
  In particular, the following infinite product is convergent:
$$ y_{j}(z,q):=w_{j}(z,q)\displaystyle\L_{q,1+(q-1)\widetilde{f}_{j,0}}\displaystyle\prod_{\ell=\mu_{j}}^{-1}e_{p^{\ell}}\left(-\ell^{-1}\widetilde{f}_{j,\ell}z^{\ell}\right)\displaystyle\prod_{\nu=-\infty}^{-1}1+(q-1)f_{j}^{>0}(q^{\nu}z,q).$$
Moreover, ${z\mapsto y_{j}(z,q)\in \mathcal{M}(\C^{*},0)}$, satisfies $\sq y_{j}(z,q)=(1+(q-1)f_{j}(z,q))y_{j}(z,q)$. To conclude the proof, it is sufficient to prove the uniform convergence $\lim\limits_{q \to 1}y(z,q)=:\widetilde{y}_{j}(z)$
in~$K$.\par 

By construction, $w_{j}(z^{-1},q)$ satisfies a linear $q$-difference equation of the form (remind that $p=1/q$)
$$\sq^{-1} w_{j}(z^{-1},q)=(1+(q-1)^{2}\b_{j}(z,q))w_{j}(z^{-1},q),$$
where $\b_{j}(z,q)\in \mathbb{B}_{d,a,\e}$ and $z\mapsto \b_{j}(z,q) \in\C\left\{z\right\}$ has constant term equals to $0$. Since ${z\mapsto w_{j}(z,q)\in \C\{z^{-1}\}}$ has constant term equal to $1$, we obtain that for every $\e'>0$, for every $z\in K$, we have for $q$ sufficiently close to $1$
$$|w_{j}(z,q)|\leq  \left|\displaystyle\prod_{\nu=-\infty}^{-1}1+(q-1)\e'\left(q^{\nu}z^{-1}\right)\right|\leq e_{q}\left(\e'|z|^{-1}\right).$$
We may check that for all $z\in \C$, $e_{q}\left(\e'|z|^{-1}\right)\leq \exp\left(\e'|z|^{-1}\right)$.
Therefore, we have the uniform convergence ${\lim\limits_{q \to 1}w_{j}(z,q)=1}$ in $K$. The uniform convergence ${\lim\limits_{q \to 1}\displaystyle\L_{q,1+(q-1)\widetilde{f}_{j,0}}=z^{\widetilde{f}_{j,0}}}$ in 
$K$ can be deduced from Page 1048 of \cite{S00}. The uniform convergence ${\lim\limits_{p \to 1}\displaystyle\prod_{\ell=\mu_{j}}^{-1}e_{p^{\ell}}\left(-\ell^{-1}\widetilde{f}_{j,\ell}z^{\ell}\right)=\displaystyle\prod_{\ell=\mu_{j}}^{-1}\exp\left(-\ell^{-1}\widetilde{f}_{j,\ell}z^{\ell}\right)}$ in $K$ is a consequence of the inequality above (\ref{eq5}) and (\ref{eq5}).
Because of~(\ref{eq19}), we obtain that for all $q>1$, for all $z\in \C$,
 $$\left|\displaystyle\prod_{\nu=-\infty}^{-1}1+(q-1)f_{j}^{>0}(q^{\nu}z,q)\right|\leq e_{q}(\a |z|)\leq \exp(\a|z|).$$ Then, the uniform convergence ${\lim\limits_{q \to 1}\displaystyle\prod_{\nu=-\infty}^{-1}1+(q-1)f_{j}^{>0}(q^{\nu}z,q)}$ in $K$ can be deduced with the dominated convergence theorem. This proves the uniform convergence ${\lim\limits_{q \to 1}y(z,q)=\widetilde{y}_{j}(z)}$
in every compact subset of $D_{d,a,\e}$.
\end{proof}

\pagebreak[3]
\begin{center}
\textbf{Step 4: Statement and proof of the main result.}
\end{center}

We are now ready to state the main result of the paper, but before, let us remind some notations. 
We still consider $$
\left\{\begin{array}{lll}
\D_{q}&=&\left(\dq -f_{m}(z,q)\right)\dots\left(\dq -f_{1}(z,q)\right)\\
\widetilde{\D}&=&\left(\d -\widetilde{f}_{m}(z)\right)\dots\left(\d -\widetilde{f}_{1}(z)\right),
\end{array} \right. 
$$
where $\widetilde{f}_{1},\dots,\widetilde{f}_{m}\in \C((z))$ have been defined in the beginning of the subsection and $f_{1},\dots,f_{m}$ have been defined before Lemma \ref{lem4}. As in $\S \ref{sec1},\S \ref{sec2}$, we are going to see the equations as systems:
$$\left\{\begin{array}{lll}
\dq Y(z,q)&=&C(z,q)Y(z,q)\\ 
 \d \widetilde{Y}(z)&=&\widetilde{C}(z)\widetilde{Y}(z).
 \end{array} \right. $$ Let ${z\mapsto \L(z,q)\in \mathrm{GL}_{m}\left(\mathcal{M}(\C^{*})\right)}$, family of diagonal matrices solution of (\ref{eq8}) that has been defined in Lemma \ref{lem3}. Let $\l\in \C^{*}$ with $\arg(\l)=d$. Let $U_{\L}^{[\l]}(z,q)$, be the fundamental solution for the system ${\dq Y(z,q)=C(z,q)Y(z,q)}$, which is defined in $\S \ref{sec22}$. We remind that for all $1\leq j\leq m$, $\widetilde{u}^{d}_{j,j}$, is a non zero meromorphic solution of $\d\widetilde{u}^{d}_{j,j}=\widetilde{S}^{d}\left(\widetilde{f}_{j}\right)\widetilde{u}^{d}_{j,j}$ and ${\left(\widetilde{u}^{d}_{j,k}(z)\right):=\widetilde{U}^{d}(z)\in \mathrm{GL}_{m}(\mathcal{A}(d-\e,d+\e ))}$, is an upper triangular matrix fundamental solution for $\d \widetilde{Y}(z)=\widetilde{C}(z)\widetilde{Y}(z)$ defined as follows:
 for all $1\leq j<k\leq m$, for all $z\in D_{d,a,\e}$, set 
$$\widetilde{u}^{d}_{j,k}(z):=\widetilde{u}^{d}_{j,j}(z)\displaystyle\int_{0}^{z} \frac{\widetilde{u}^{d}_{j+1,k}(t)}{\widetilde{u}^{d}_{j,j}(t)}\frac{dt}{t}.$$
After replacing $a$ by a smaller positive real number, we may reduce to the case where for all $1\leq j\leq m$, $\widetilde{u}^{d}_{j,j}$ does not vanish on $D_{d,a,\e}$.
\pagebreak[3]
\begin{theo}\label{theo3}
 We have the uniform convergence
 $$\lim\limits_{q \to 1}U_{\L}^{[\l]}(z,q)=\widetilde{U}^{d}(z),$$ 
in every the compact subset of $D_{d,a,\e}$. 
\end{theo}

We want to apply Theorem \ref{theo2} to prove Theorem \ref{theo3}. We already know that assumptions \textbf{(H1)} to \textbf{(H3)} are satisfied. Because of Lemma \ref{lem3},  hypothesis \textbf{(H4)} is also satisfied. We have to check that the assumption \textbf{(H5)} is satisfied. This is the goal of the following lemma.\par

\pagebreak[3]
\begin{lem}\label{lem5}
For all $1\leq j <m$, we have
$\dfrac{u_{\L,j+1,j+1}^{[\l]}}{z\sq\left(u_{\L,j,j}^{[\l]}\right)}\in \mathbb{B}_{d,a,\e}$. 
\end{lem}

\begin{proof}
Let us fix $1\leq j <m$. Because of (\ref{eq17}), we obtain 
\begin{equation}\label{eq14}
\sq\left(\dfrac{u_{\L,j+1,j+1}^{[\l]}}{z\sq \left(u_{\L,j,j}^{[\l]}\right)}\right)=\frac{1+(q-1)\sq\left(f_{j} \right)}{1+(q-1)\sq^{2}\left(f_{j} \right)}\frac{1+(q-1)f_{j+1}^{>0}}{1+(q-1)f_{j}^{>0}}\left(1+(q-1)\left(\widetilde{f}_{j+1}^{\leq 0}-\widetilde{f}_{j}^{\leq 0} -1\right)\right) \dfrac{u_{\L,j+1,j+1}^{[\l]}}{z\sq\left(u_{\L,j,j}^{[\l]}\right)}.
\end{equation}
We remind, see Lemma \ref{lem6}, that for all $z\in D_{d,a,\e}$,  $$\lim\limits_{\substack{ x \to 0^{+}\\ x\in \R_{>0}}}\left|\frac{\exp\left(\widetilde{f}_{j+1,0}\log(xz)\right)\displaystyle\prod_{\ell=\mu_{j+1}}^{-1}\exp\left(-\ell^{-1}\widetilde{f}_{j+1,\ell}(xz)^{\ell}\right)}{\exp\left(\left(\widetilde{f}_{j,0}+1\right)\log(xz)\right)\displaystyle\prod_{\ell=\mu_{j}}^{-1}\exp\left(-\ell^{-1}\widetilde{f}_{j,\ell}(xz)^{\ell}\right)}\right|=0.$$
Then, there exist $a'\in ]0,a[$ and $M>0$, such that for every $z\in D_{d,a',\e}$, for all $q$ sufficiently close to $1$, 
\begin{equation}\label{eq18}
\left| 1+(q-1)\left(\widetilde{f}_{j+1}^{\leq 0}-\widetilde{f}_{j}^{\leq 0} -1\right)\right|>1+(q-1)M|z|^{\mu_{j}}.
\end{equation}
The triangular inequality gives us $$\left|\frac{1+(q-1)\sq\left(f_{j} \right)}{1+(q-1)\sq^{2}\left(f_{j} \right)}\frac{1+(q-1)f_{j+1}^{>0}}{1+(q-1)f_{j}^{>0}}\right|\geq \left|\frac{1-(q-1)\left|\sq\left(f_{j} \right)\right|}{1+(q-1)\left|\sq^{2}\left(f_{j} \right)\right|}\frac{1-(q-1)\left|f_{j+1}^{>0}\right|}{1+(q-1)\left|f_{j}^{>0}\right|}\right|.$$
Using this inequality, (\ref{eq11}), (\ref{eq14}), and (\ref{eq18}), we find that there exists $a''\in ]0,a'[$, such that for every $z\in D_{d,a'',\e}$, for all $q$ sufficiently close to $1$, 
\begin{equation}\label{eq16}
\left|\sq\left(\dfrac{u_{\L,j+1,j+1}^{[\l]}}{z\sq\left(u_{\L,j,j}^{[\l]}\right)}\right)\right|> \left|\dfrac{u_{\L,j+1,j+1}^{[\l]}}{z\sq\left(u_{\L,j,j}^{[\l]}\right)}\right|.
\end{equation}
We remind that $\widetilde{u}^{d}_{j,j}$ does not vanish on $D_{d,a,\e}$. Then, with Lemma \ref{lem3}, we obtain that we have the uniform convergence ${\lim\limits_{q \to 1}\dfrac{u_{\L,j+1,j+1}^{[\l]}}{z \sq\left(u_{\L,j,j}^{[\l]}\right)}=\dfrac{\widetilde{u}_{j+1,j+1}^{d}(z)}{z\widetilde{u}_{j,j}^{d}(z)}\in \mathcal{A}(d-\e,d+\e )}$,
in every compact subset of $D_{d,a,\e}$. If we combine this fact and (\ref{eq16}), we find $\dfrac{u_{\L,j+1,j+1}^{[\l]}}{zu_{\L,j,j}^{[\l]}}\in \mathbb{B}_{d,a,\e}$.
This proves the lemma.
\end{proof}


\bibliographystyle{alpha}
\bibliography{biblio}

\end{document}